\documentclass[11pt,a4paper]{article}

\usepackage[english,french]{babel}
\usepackage{amsfonts}
\usepackage{amsmath,amssymb,amscd}
\usepackage{mathrsfs}
\usepackage{amsthm}
\usepackage{a4wide}
\usepackage[T1]{fontenc}
\usepackage{newlfont}

\newtheorem{thm}{Th\'eor\`eme}[section]
\newtheorem{dfn}[thm]{D\'efinition}

\newtheorem{exem}[thm]{Exemple}
\newtheorem{rem}[thm]{Remarque}
\newtheorem{rems}[thm]{Remarques}
\newtheorem{prop}[thm]{Proposition}
\newtheorem{lem}[thm]{Lemme}
\newtheorem{coro}[thm]{Corollaire}

\newtheorem*{thm*}{Th\'eor\`eme}
\newtheorem*{dfn*}{D\'efinition}
\newtheorem*{defiprop*}{D\'efinition-Proposition}
\newtheorem*{exem*}{Exemple}
\newtheorem*{rem*}{Remarque}
\newtheorem*{rems*}{Remarques}
\newtheorem*{prop*}{Proposition}
\newtheorem*{lem*}{Lemme}
\newtheorem*{coro*}{Corollaire}

\begin{document}

\title{Applications des immeubles en thÈorie des reprÈsentations}
\author{St\'ephane Gaussent, Nicole Bardy, Cyril Charignon  et Guy Rousseau}
\maketitle
\bigskip

\section{Introduction}
\label{seIntro}

Ce rapport a pour but d'exposer le lien entre la th\'eorie des repr\'esentations d'un groupe semi-simple et la g\'eom\'etrie d'un immeuble affine. Celui-ci est en fait associ\'e au groupe semi-simple dual de Langlands du pr\'ec\'edent. Comme nous parlerons plus de cette g\'eom\'etrie, on privil\'egie ce dernier groupe dans les notations qui suivent.
$$
\begin{array}{cl}
G & \hbox{groupe semi-simple sur le corps des complexes } \mathbb C \hbox{, associ\'e \`a }(X,Y,\Phi,\Phi^\vee) \\
T & \hbox{tore maximal dans } G\\
X & \hbox{r\'eseau des caract\`eres de } T\\
Y & \hbox{r\'eseau des cocaract\`eres de } T\\
\Phi & \hbox{syst\`eme de racines de } (G,T), \hbox{ contenu dans } X  \hbox{ et suppos\'e irr\'eductible}\\
\{\alpha_i\}_{i\in I} & \hbox{un choix de racines simples dans } \Phi\\
P^\vee & \hbox{r\'eseau des copoids, dual de } Q = \oplus \mathbb Z\alpha_i\\
\Phi^\vee & \hbox{syst\`eme des coracines de } (G,T), \hbox{ contenu dans } Y\\
Q^\vee & \hbox{r\'eseau des coracines de } G, Q^\vee = \oplus \mathbb Z\alpha^\vee_i\subset Y\subset P^\vee\\
G^\vee & \hbox{dual de Langlands de } G\hbox{, groupe semi-simple complexe  associ\'e \`a }(Y,X,\Phi^\vee,\Phi)\\
V(\lambda) & \hbox{repr\'esentation irr\'eductible de } G^\vee \hbox{ de plus haut poids } \lambda\in Y^+=\{\lambda\in Y\mid\alpha_i(\lambda)\geq0\}\\
\mathcal I & \hbox{immeuble de Bruhat-Tits associ\'e \`a } G \hbox{ et au corps des s\'eries de Laurent } \mathscr K = \mathbb C(\!(t)\!)\\

\end{array}
$$

Nous reprenons ici, en grande partie, les expos\'es d'un groupe de travail qui a eu lieu \`a Nancy en 2008-2009 sur le th\'eor\`eme de <<saturation>>. Ce th\'eor\`eme a \'et\'e d\'emontr\'e par Kapovich et Millson \cite{KM}, \`a la suite d'une s\'erie d'articles avec Leeb (\cite{KLM1}, \cite{KLM2}, \cite{KLM3},...) ; il s'\'enonce comme suit :
\begin{thm}
\label{thSat}
On note $k = k_\Phi$ le plus petit multiple commun des coefficients de la plus grande racine de $\Phi$.  Soient $\lambda,\mu$ et $\nu$ des copoids dominants tels que $\lambda+\mu+\nu \in Q^\vee$. S'il existe $N\in\mathbb N^*$ tel que $\big (V(N\lambda)\otimes V(N\mu)\otimes V(N\nu)\big )^{G^\vee} \ne \{0\}$ alors $\big (V(k^2\lambda)\otimes V(k^2\mu)\otimes V(k^2\nu)\big )^{G^\vee} \ne \{0\}$.
\end{thm}

La conjecture qu'ils formulent est que le r\'esultat reste vrai si on remplace $k^2$ par $k$ et m\^eme si on le remplace par $1$ ou $2$, selon que les racines de $\Phi$ sont toutes de m\^eme longueur ou non. Pour $G$ de type $A$, $k=1$ et on a bien une saturation du c\^one de Littlewood-Richardson (form\'e des $(\lambda,\mu,\nu)$ tels que $\big (V(\lambda)\otimes V(\mu)\otimes V(\nu)\big )^{G^\vee} \ne \{0\}$). Dans ce cas-l\`a, le r\'esultat a \'et\'e d\'emontr\'e par Knutson et Tao avec le mod\`ele du nid d'abeilles \cite{KT}.

L'importance de ce th\'eor\`eme tient au fait que la non nullit\'e de  $\big (V(\lambda')\otimes V(\mu')\otimes V(\nu')\big )^{G^\vee}$ traduit l'existence d'une sous-repr\'esentation de $V(\lambda')\otimes V(\mu')$ isomorphe au dual $V(\nu'^*)$ de $V(\nu')$ et la multiplicit\'e de $V(\nu'^*)$ est la dimension de cet espace. On a donc ainsi des renseignements sur la d\'ecomposition en facteurs irr\'eductibles du produit tensoriel de ces deux repr\'esentations irr\'eductibles.

\medskip
\noindent{\bf Esquisse de la preuve. } En gros, la preuve se d\'ecanule en les points suivants :

{\bf Etape 1)} $\big (V(N\lambda)\otimes V(N\mu)\otimes V(N\nu)\big )^{G^\vee} \ne \{0\}$ implique qu'il existe un triangle (g\'eod\'esique) $T(0,A,B)$ dans $\mathcal I$, de longueurs de c\^ot\'es $N\lambda, N\mu, N\nu$.

{\bf Etape 2)} Il existe une application, appel\'ee <<application de Gauss>>
, qui \`a $T(0,A,B)$ associe une configuration semi-stable $((N\lambda, \xi_1), (N\mu, \xi_2), (N\nu, \xi_3))$ de points pond\'er\'es de $\partial_\infty \mathcal I$ (le bord visuel de $\mathcal I$).

{\bf Etape 3)} L'ensemble des configurations semi-stables est satur\'e, ainsi $((\lambda, \xi_1), (\mu, \xi_2), (\nu, \xi_3))$ est toujours semi-stable. On veut maintenant inverser l'application de Gauss
. Pour cela on construit \`a partir de $((\lambda, \xi_1), (\mu, \xi_2), (\nu, \xi_3))$ une application $\phi:\mathcal I\to\mathcal I$. En fait, on montre qu'il existe un point fixe not\'e $x_0$ de $\phi$, d'o\`u un triangle $T(x_0,x_1,x_2)$ dans $\mathcal I$, de longueurs de c\^ot\'es $\lambda, \mu, \nu$.

{\bf Etape 4)} La condition $\lambda+\mu+\nu \in Q^\vee$ implique que $x_0,x_1,x_2$ sont des sommets de $\mathcal I$. On r\'etracte dans un appartement contenant $x_0$ et par rapport \`a une alc\^ove ${\mathfrak a}$ qui contient ce point. On dilate par $k$, les sommets deviennent des sommets sp\'eciaux et on obtient un polygone $P(0,a,a_1,...,a_n,b,0)$ form\'e de deux segments $[0,a]$ et $[b,0]$ de longueurs de c\^ot\'es $k\lambda$ et $k\nu$ et d'une ligne bris\'ee qui est en fait un chemin de Hecke par rapport \`a  ${\mathfrak a}$ de type $k\mu$.

{\bf Etape 5)} On replie ce polygone dans la chambre fondamentale, ce qui en donne un form\'e par les segments $[0,k\lambda]$, $[0,k\nu^*]$ et un chemin de Hecke de type $k\mu$ contenu dans la chambre fondamentale (ici, $\nu^* = -w_0 \nu$). Malheureusement, un chemin de Hecke n'est pas forc\'ement LS. Mais, modulo un petit ajustement, en dilatant de nouveau, on arrive \`a un chemin LS. On conclut par un th\'eor\`eme de Littelmann que $\big (V(k^2\lambda)\otimes V(k^2\mu)\otimes V(k^2\nu)\big )^{G^\vee} \ne \{0\}$.

\medskip
Les expos\'es qui suivent s'inspirent largement de \cite{KM}.
La section \ref{seCheLS} introduit les notions de chemins LS et Hecke dues \`a Littelmann et Kapovich-Millson, et explique leur importance en th\'eorie des repr\'esentations. Dans la section \ref{seLSMV}, non essentielle pour le th\'eor\`eme de saturation, on interpr\`ete dans l'immeuble $\mathcal I$, gr\^ace \`a des galeries dites LS, les cycles de Mirkovi\'c-Vilonen qui interviennent \'egalement en th\'eorie des repr\'esentations, cf. \cite{GL}.
La section \ref{seDepli} fournit une preuve originale (enti\`erement immobili\`ere) d'une caract\'erisation par Kapovich et Millson des images dans $\mathbb A$ des triangles g\'eod\'esiques de $\mathcal I$ par certaines r\'etractions; on en d\'eduit l'\'etape 1) ci-dessus. Dans la section \ref{seGauss} on introduit le bord visuel de $\mathcal I$ et la notion de configuration semi-stable; on d\'efinit l'application de Gauss et on montre qu'elle a un point fixe; ce sont les \'etapes 2) et 3). On conclut enfin, dans la derni\`ere section, apr\`es avoir accompli les \'etapes 4) et 5).

La section \ref{seIntro} est l'expos\'e introductif de S. Gaussent, la section \ref{seCheLS} correspond \`a des expos\'es de N. Bardy et G. Rousseau. La section \ref{seLSMV} est un expos\'e, plus ancien, de S. Gaussent. 
La section \ref{seDepli}  regroupe des expos\'es de C. Charignon avec des compl\'ements finaux de S. Gaussent. Enfin les sections \ref{seGauss} et \ref{seSat} correspondent respectivement \`a des expos\'es de S. Gaussent et de G. Rousseau.

\newpage
\section{Diff\'erentes d\'efinitions des chemins LS}
\label{seCheLS}

\subsection{Appartement dans $\mathcal I$}
\label{sseAppart}

On note $\mathbb A$ l'appartement t\'emoin dans $\mathcal I$, il s'agit d'un espace affine euclidien sous $V=Y_{\mathbb R} = Y\otimes \mathbb R$ (souvent identifi\'e \`a V). De plus, on d\'esigne par $C^v = \{x\in V\mid \alpha_i(x)\geqslant 0,\ \forall i\in I\}$ la chambre de Weyl fondamentale; la plupart des notions introduites ci-dessous d\'ependent du choix de cette chambre.
Le groupe de Weyl fini $W^v$ agit isom\'etriquement sur $V$ avec  $C^v$ comme domaine fondamental. Il est engendr\'e par les r\'eflexions $r_\alpha$, pour $\alpha\in\Phi$, o\`u $r_\alpha$est la sym\'etrie orthogonale par rapport \`a l'hyperplan Ker$\alpha$ (mur vectoriel).

Pour tout $v\in V$, il existe un unique $v_0\in C^v\cap W^vv$, on d\'efinit la projection sur $C^v$ en posant $v_0 = pr_{C^v}(v)$. Si $v\in C^v$, on note $v^*=pr_{C^v}(-v)=-w_0.v$ (si $w_0$ est l'\'el\'ement de plus grande longueur de $W^v$). De plus, pour tous $x,y\in\mathbb A $, on pose $d_{C^v}(x,y) = pr_{C^v}(y-x)\in C^v$. On dit que $d_{C^v}(x,y)$ est la longueur du segment orient\'e $[x,y]$.

L'ensemble $\mathcal M$ des murs de $\mathbb A$ est en bijection avec $\Phi\times\mathbb Z$, il s'agit des hyperplans
$$
M(\alpha,k)= \{x\in\mathbb A\mid \alpha(x) + k =0\}.
$$ La r\'eflexion $s_M=s_{\alpha,k}$ associ\'ee au mur $M=M(\alpha,k)$ respecte l'ensemble $\mathcal M$ des murs et ces r\'eflexions engendrent le groupe de Weyl affine : $W^a = W^v\ltimes Q^\vee$. 
 Un mur de $\mathbb A$ d\'etermine deux demi-espaces ferm\'es de $\mathbb A$ (appel\'es demi-appartements) dont il constitue le bord.

Une alc\^ove dans $\mathbb A$ est l'adh\'erence d'une composante connexe du compl\'ementaire des murs; c'est un domaine fondamental pour l'action de $W^a$. Un sommet de $\mathcal A$ est un sommet $x$ d'une alc\^ove, il est sp\'ecial si, $\forall\alpha\in\Phi$, on a $\alpha(x)\in\mathbb Z$ (i.e. si $x\in P^\vee$).

On sera amen\'e plus tard \`a consid\'erer des appartements  construits comme ci-dessus, en rempla\c{c}ant $\mathbb Z$ par un autre sous-groupe discret de $\mathbb R$, par exemple $\{0\}$. Dans ce dernier cas, $W^a=W^v$, l'appartement est dit vectoriel et ses alc\^oves sont aussi ses chambres de Weyl.

\subsection{Chemins polygonaux parcourus \`a vitesse constante}

Si $\lambda\in C^v$, un $\lambda-$chemin ou un chemin de type $\lambda$ est un chemin lin\'eaire par morceaux $\pi : [0,1]\to \mathbb A$ tel que pour tout $t$ (sauf un nombre fini), $pr_{C^v}(\pi'(t)) = \lambda$. Les d\'eriv\'ees \`a droite et \`a gauche en $t$, $\pi'_+(t)$ et $\pi'_-(t)$ existent tout le temps, mais sont parfois non identiques. La somme des longueurs des segments constituant le $\lambda-$chemin $\pi$ est $\lambda$.
Dans Kapovich et Millson, un tel chemin s'appelle un chemin billard (voir paragraphe \ref{sseCheKM}). 

Un $\lambda-$chemin s'\'ecrit $\pi(\lambda,\pi_0,\mathbf w,\mathbf a)$, o\`u $\mathbf w = (w_1,...,w_m)\in (W^v)^m$, $\mathbf a = (a_0 =0 <a_1<\cdots < a_m = 1)$ et
$$
\pi(t) = \pi_0 + \sum_{i=1}^{j-1} (a_i-a_{i-1})w_i(\lambda) + (t-a_{j-1})w_j(\lambda)
$$ si $a_{j-1}\leqslant t\leqslant a_j$.

\subsection{Galeries}

Les chambres de Weyl de $V$ sont les transform\'es de $C^v$ par $W^v$. Une cloison est une facette de codimension $1$ d'une chambre, son type est l'\'el\'ement $i\in I$ tel qu'elle soit conjugu\'ee par $W^v$ \`a $C^v\cap{\mathrm Ker}\alpha_i$. Deux chambres sont dites mitoyennes si elles ont une cloison en commun (elles peuvent Ítre Ègales).

Une galerie de chambres de $C$ \`a $C'$ est une suite $\Gamma=(C_0=C,C_1,...,C_n=C')$ de chambres telle que $C_{i-1}$ et $C_{i}$ soient mitoyennes, pour $1\leq i\leq n$. Cette galerie est dite minimale ou tendue si la longueur $n$ est minimale; alors $n$ est la distance de $C$ \`a $C'$. La suite des types de cloisons dans $C_{i-1}\cap C_{i}$ est un type de cette galerie. Plier $\Gamma$ (au niveau $j$), c'est avoir $C_{j-1}\ne C_j$ et remplacer $C_j,...,C_n$ par leurs images par la r\'eflexion par rapport au mur (vectoriel) contenant la cloison $C_{j-1}\cap C_j$.

Comme une alc\^ove est un domaine fondamental pour $W^a$, toutes les d\'efinitions pr\'ec\'edentes peuvent \^etre r\'ep\'et\'ees pour les alc\^oves. On d\'efinit ainsi dans $\mathbb A$ des cloisons d'alc\^ove et des galeries d'alc\^oves. Dans ce cadre les types correspondent aux cloisons d'une alc\^ove fondamentale; comme $\Phi$ est irr\'eductible, ils sont index\'es par $I\cup\{0\}$.

\subsection{Ordre de Bruhat-Chevalley}

Le groupe $W^v$ est un groupe de Coxeter pour le syst\`eme de g\'en\'erateurs $\{r_i=r_{\alpha_i}\mid i\in I\}$. Tout $w\in W^v$ peut se d\'ecomposer sous la forme $w=r_{i_1}.....r_{i_n}$; la longueur $\ell(w)$ de $w$ est le minimum des $n$ possibles, une d\'ecomposition avec $\ell(w)$ termes est alors dite r\'eduite.

On a le r\'esultat classique suivant.
\begin{prop}
Dans $W^v$, les assertions suivantes sont \'equivalentes :
\begin{enumerate}
\item $w'\leq w$ (ordre de Bruhat-Chevalley) ;
\item il existe une suite $w'=w_0,w_1,...,w_n=w$ et des racines $\beta_i$ tels que $w_{i+1} = w_i r_{\beta_i}$ et $\ell(w_{i+1})>\ell (w_i)$ ;
\item idem avec les $r_{\beta_i}$ \`a gauche ;
\item idem avec $\ell(w_{i+1}) = \ell (w_i) + 1$ ;
\item $w'$ est le produit d'une sous-expression d'une d\'ecomposition r\'eduite de $w$ ;
\item il existe une galerie minimale de $C^v$ \`a $wC^v$ qui donne une galerie de $C^v$ \`a $w'C^v$ par des pliages successifs.
\end{enumerate}
\end{prop}

On d\'efinit un ordre dans $W^v/W^v_\lambda$, o\`u $W^v_\lambda = \{w\in W^v\mid w(\lambda) = \lambda\} = \langle r_{\alpha_i}\mid \alpha_i(\lambda) = 0\rangle$. Etant donn\'ee une classe $\tilde w$, il existe un unique $\tilde w_0\in\tilde w$ de longueur minimale not\'ee $\ell_\lambda(\tilde w)$. On a $\tilde w_0 \leq w$, pour tout $w\in\tilde w$ et m\^eme $w = \tilde w_0 u$, $u\in W^v_\lambda$ avec $\ell (w) = \ell (\tilde w_0) + \ell (u)$. De plus, $\tilde w_0C^v$ est la projection de $C^v$ sur $\tilde w\lambda$, i.e. la chambre contenant $\tilde w\lambda$ la plus proche de $C^v$. On d\'efinit
$$
\tilde w'\leq \tilde w \quad\hbox{ par } \quad\tilde w'_0\leq\tilde w_0 \ .
$$

\begin{prop} les conditions suivantes sont \'equivalentes :
\begin{enumerate}
\item $\tilde w'\leq \tilde w$ ;
\item $\exists w\in\tilde w$, $\exists w'\in\tilde w'$, $w'\leq w$ ;
\item il existe une galerie minimale de $C^v$ \`a $\tilde w\lambda$ qui donne une galerie de $C^v$ \`a $\tilde w'\lambda$ par des pliages successifs.
\end{enumerate}
\end{prop}

La preuve est laiss\'ee \`a la sagacit\'e du lecteur.

\subsection{Les chemins \`a la Kapovich et Millson}
\label{sseCheKM}


Un chemin $\pi:[0,1]\to\mathbb A$ est un vrai billard si $\pi$ est lin\'eaire par morceaux et si, pour tout $t$, $\pi'_+(t)\in W_{\pi(t)}^v\pi'_-(t)$, o\`u $W_{\pi(t)}^v = \langle r_\alpha\mid \alpha\in\Phi, \ \alpha(\pi(t))\in\mathbb Z\rangle$. Un chemin qui est un vrai billard est un $\lambda-$chemin pour $\lambda = pr_{C^v}(\pi'_\pm(t))$.

Si $\pi$ est un $\lambda-$chemin, on note $w_\pm(t)$ l'\'el\'ement de $W^v$ de plus petite longueur tel que $\pi'_\pm(t) = w_\pm(t)\lambda$. Le chemin lin\'eaire par morceaux $\pi$ est pli\'e positivement si, pour tout $t$, il existe une $W^v-$cha\^{\i}ne de $\pi'_-(t)$ \`a $\pi'_+(t)$, c'est-\`a-dire qu'il existe une suite de vecteurs $\pi'_-(t) = \eta_0,\eta_1,...,\eta_m = \pi'_+(t)$ et des racines positives $\beta_1,...,\beta_m$ telles que
\begin{description}
\item[(H1)] $\quad r_{\beta_i}(\eta_{i-1}) = \eta_i $
\item[(H2)] $\quad \beta_i(\eta_{i-1}) <0\ $.
\end{description}

\begin{lem}
Un $\lambda-$chemin $\pi$ est pli\'e positivement si, et seulement si, pour tout $t$, $w_+(t)\leq w_-(t)$.
\end{lem}
\begin{proof} On a $\eta_0 = \pi'_-(t) = w_-(t)\lambda$ et $\eta_i = w_i\lambda$, o\`u on a pos\'e $w_i = r_{\beta_i}\cdots r_{\beta_1} w_-(t)$. Or $\beta_i(\eta_{i-1}) <0$ implique que $\beta_i(w_{i-1}C^v) \leqslant 0$ ce qui \'equivaut \`a $\ell (w_i = r_{\beta_i}w_{i-1}) < \ell(w_{i-1})$. Donc $w_-(t) = w_0>w_1>\cdots >w_m \geq w_+(t)$ car $\pi'_+(t) = w_m\lambda$.

R\'eciproquement, si $w_+(t) \leq w_-(t)$, d'apr\`es l'une des caract\'erisations de l'ordre de Bruhat, il existe des racines positives $\beta_1,...,\beta_m$ telles que si on pose $w_0 = w_-(t)$, $\eta_0 = \pi'_-(t)$ et, pour $1\leq i\leq m$, $w_i = r_{\beta_i}\cdots r_{\beta_1} w_-(t)$, $\eta_i = w_i\lambda$, on a $w_m = w_+(t) $, $\eta_m = \pi'_+(t)$ et  $\ell(w_i)<\ell(w_{i-1})$
. D'o\`u, {\bf (H1)}.

Pour {\bf (H2)},
$$
\ell(r_{\beta_i}w_{i-1})<\ell(w_{i-1}) \Rightarrow \beta_i(w_{i-1}C^v) \leqslant 0 \Rightarrow \beta_i(\eta_{i-1}) \leqslant 0 \ .
$$ Si c'est une in\'egalit\'e stricte, c'est bon. Sinon, $\eta_{i-1} = \eta_i$ et on oublie $i$ dans la liste.
\end{proof}

\bigskip
Un chemin lin\'eaire par morceaux $\pi$ est de Hecke (par rapport \`a $-C^v$) si, pour tout $t$, il existe une $W^v_{\pi(t)}-$cha\^{\i}ne de $\pi'_-(t)$ \`a $\pi'_+(t)$, c'est \`a dire une  $W^v-$cha\^{\i}ne comme ci-dessus v\'erifiant de plus:
\begin{description}
\item[(H3)] $r_{\beta_i}\in W^v_{\pi(t)}$, i.e. $\beta_i(\pi(t))\in\mathbb Z$: $\pi(t)$ est dans un mur de direction Ker$\beta_i$.
\end{description}
Alors, Hecke implique vrai billard et pli\'e positivement, mais la r\'eciproque est fausse en g\'en\'eral (on peut trouver un contre-exemple en type $G_2$). 
Si les points anguleux de $\pi$ sont des sommets sp\'eciaux, la condition {\bf (H3)} est automatiquement v\'erifi\'ee; alors billard et pli\'e positivement impliquent Hecke.

\subsection{Les chemins LS}

Un chemin de {\it Lakshmibai-Seshadri} (ou {\it LS}) de type $\lambda\in C^v$ est un $\lambda-$chemin $\pi=\pi(\lambda,\pi_0,{\mathbf w},{\mathbf a})$ 
 tel que : pour tout $j=1,\dots,m-1$, il existe une $a_j-$cha\^{\i}ne de $w_j$ \`a
$w_{j+1}$ i.e. il existe une suite $\beta_{j,1},\dots,\beta_{j,s_j}$ de racines positives telle que, si on pose $\sigma_{j,0}=w_{j}\;,\;\sigma_{j,1}=r_{\beta_{j,1}}w_j\;,\dots,\;
\sigma_{j,s_j}=r_{\beta_{j,s_j}}\dots r_{\beta_{j,1}}w_j$, on a $\sigma_{j,s_j}=w_{j+1}$ et

\begin{description}
\item[(LS0)] $\pi_0$ est un sommet sp\'ecial (i.e. $\pi_0\in P^\vee$) et $\lambda\in P^{\vee+}=P^\vee\cap C^v$.
\item[(LS1)] $\sigma_{j,i}<\sigma_{j,i-1}$,\quad dans $W^v/W^v_\lambda$ ;
\item[(LS2)] $a_j\beta_{j,i}(\sigma_{j,i}(\lambda))\in\mathbb Z$ ;
\item[(LS3)] $\ell_\lambda(\sigma_{j,i})=\ell_\lambda(\sigma_{j,i-1})-1$.
\end{description}
\noindent En fait, Littelmann (\cite{L1}, \cite{L2}) consid\`ere des chemins LS normalis\'es, i.e. avec $\pi_0=0$.

\bigskip
\noindent{\bf N.B. } 1) {\bf (LS0)} + {\bf (LS2)} $\Rightarrow a_j\in \mathbb Q$ (si $w_j \ne w_{j-1}$).

2) On sait que $\pi(1)-\pi(0)-\lambda\in -Q^{\vee+}=-\sum\,\mathbb N\alpha^\vee_i$.

\begin{prop}
\label{prLS1et2}
Si {\bf (LS0)} est v\'erifi\'e alors $\quad $ {\bf (LS1)} $+$ {\bf (LS2)} $\iff$ Hecke.
\end{prop}
\noindent La preuve de cette proposition se trouve dans \cite{GR}.


\subsection{Les op\'erateurs $e_\alpha$ et $f_\alpha$}
\label{sseOpera}

Soient $\pi:[0,1]\to\mathbb A$ un chemin lin\'eaire par morceaux d'origine $\pi(0)=0$ et $\alpha$ une racine simple.

On consid\`ere $Q={\mathrm Inf}(\alpha\circ\pi([0,1])\cap\mathbb Z)$. Si $Q=0$, le chemin $e_\alpha\pi$ n'est pas d\'efini. Si $Q<0$, soient $q={\mathrm Inf}\{t\in[0,1]\mid \alpha\circ\pi(t)=Q\}$ et $y={\mathrm Sup}\{t\in[0,q]\mid\alpha\circ\pi(t)=Q+1\}$; le chemin $e_\alpha\pi$ est la concat\'enation de $\pi\vert_{[0,y]}$, d'un sym\'etrique de $\pi\vert_{[y,q]}$ et d'un translat\'e de $\pi\vert_{[q,1]}$. Plus pr\'ecis\'ement $e_\alpha\pi(t)=\pi(t)$ pour $t\in[0,y]$, $e_\alpha\pi(t)=s_{\alpha,Q+1}(\pi(t))$ pour $t\in[y,q]$ et $e_\alpha\pi(t)=\pi(t)-\pi(q)+e_\alpha\pi(q)$ pour $t\in[q,1]$.

De m\^eme on consid\`ere la partie enti\`ere $P$ de $\alpha\circ\pi(1)-Q$. Si $P\leq 0$, le chemin $f_\alpha\pi$ n'est pas d\'efini. Si $P\geq1$, soient $p={\mathrm Sup}\{t\in[0,1]\mid \alpha\circ\pi(t)=Q\}$ et $x={\mathrm Inf}\{t\in[p,1]\mid\alpha\circ\pi(t)=Q+1\}$; le chemin $f_\alpha\pi$ est la concat\'enation de $\pi\vert_{[0,p]}$, d'un sym\'etrique de $\pi\vert_{[p,x]}$ et d'un translat\'e de $\pi\vert_{[x,1]}$. Plus pr\'ecis\'ement $f_\alpha\pi(t)=\pi(t)$ pour $t\in[0,p]$, $f_\alpha\pi(t)=s_{\alpha,Q}(\pi(t))$ pour $t\in[p,x]$ et $f_\alpha\pi(t)=\pi(t)-\pi(x)+f_\alpha\pi(x)$ pour $t\in[x,1]$.

\bigskip
On a les propri\'et\'es assez faciles suivantes:

\begin{itemize}
\item Si $e_\alpha\pi$ est d\'efini, on a $e_\alpha\pi(1)=\pi(1)+\alpha^\vee$.
\item Si $f_\alpha\pi$ est d\'efini, on a $f_\alpha\pi(1)=\pi(1)-\alpha^\vee$.
\item $(e_\alpha)^n\pi$ est d\'efini si et seulement si  $n\leq -Q$.
\item $(f_\alpha)^n\pi$ est d\'efini si et seulement si  $n\leq P$.
\item Si $\pi(1)\in P^\vee$, on a $P+Q=\alpha(\pi(1))$.
\item $e_\alpha\pi=\pi' \Longleftrightarrow f_\alpha\pi'=\pi$
\item Si $\pi([0,1])\subset C^v$, aucun $e_\alpha\pi$ n'est d\'efini.
\end{itemize}

On notera l'analogie des propri\'et\'es 3, 4 et 5 avec celles des bases des repr\'esentations de $SL_2$.

Pour $\lambda\in P^{\vee+}$, soit $\pi_\lambda$ le segment $[0,\lambda]$, i.e. $\pi_\lambda(t)=t\lambda$. On montre alors le r\'esultat, plus difficile, suivant \cite{L1}:

\begin{prop}
\label{prLSop}
Un chemin $\pi$ d'origine $0$ est un chemin LS (normalis\'e) de type $\lambda$ si et seulement si on peut l'\'ecrire sous la forme $\pi=f_{\beta_1}f_{\beta_2}...f_{\beta_r}\pi_\lambda$ avec $r\in\mathbb N$ et $\beta_1,...,\beta_r$ des racines simples.
\end{prop}

\subsection{Application aux repr\'esentations de $G^\vee$}
\label{sseRepr}

Soit $\lambda\in Y^+=Y\cap C^v$. La repr\'esentation irr\'eductible de plus haut poids $\lambda$ est de dimension finie. L'action du tore maximal $T^\vee$ de $G^\vee$ est diagonalisable avec des poids $\mu\in\lambda-\sum_{i\in I}\,\mathbb N\alpha_i^\vee$ (contenu dans le r\'eseau  $Y$ des caract\`eres de $T^\vee$). Il est important de conna\^{\i}tre les poids qui interviennent et leur multiplicit\'e, c'est ce que donne le th\'eor\`eme suivant \cite{L1}.

\begin{thm}[Formule des caract\`eres de Littelmann]
\label{thCaract}
La multiplicit\'e de $\mu$ dans $V(\lambda)$ est le nombre de chemins LS (normalis\'es) de type $\lambda$ et d'extr\'emit\'e $\mu$.
\end{thm}

La formule des caract\`eres de Weyl, pr\'ec\'edemment connue, a le d\'esavantage d'exprimer la multiplicit\'e comme une somme d'entiers relatifs, ce qui rend plus difficile de voir si elle est non nulle.

Le th\'eor\`eme suivant de Littelmann \cite{L1}  est fondamental pour la preuve du th\'eor\`eme de saturation.

\begin{thm}[R\`egle de d\'ecomposition \`a la Littlewood-Richardson]
\label{thDecomp}
Soient $\lambda,\mu$ et $\nu$ des copoids dominants de $G$ (i.e. des \'el\'ements de $Y^+\subset P^{\vee+}$). Alors, $\big (V(\lambda)\otimes V(\mu)\otimes V(\nu)\big )^{G^\vee} \ne \{0\}$ si, et seulement si, il existe un chemin LS normalis\'e $\pi$ de type $\mu$ tel que $\lambda +\pi(1) = \nu^*$ et, pour tout $t\in [0,1]$, $\lambda + \pi(t) \in C^v$.
\end{thm}

{\bf N.B. } Dans ce cas on a donc $\lambda+\mu+\nu\in Q^\vee$.

\bigskip\bigskip\bigskip
\section{Chemins LS et Cycles de Mirkovi\'c-Vilonen}
\label{seLSMV}


\subsection{Grassmannienne affine et cycles de Mirkovi\'c-Vilonen}
\label{sseGrass}

Le corps $\mathscr K=\mathbb C(\!(t)\!)$ est complet pour la valuation (discr\`ete) des s\'eries de Laurent formelles, son anneau d'entiers est $\mathcal O=\mathbb C[[t]]$ et son corps r\'esiduel $\mathbb C$.

On consid\`ere le groupe alg\'ebrique semi-simple complexe $G$, son tore maximal $T$, le normalisateur de celui-ci $N$ et son sous-groupe de Borel $B$ (associ\'e \`a $T$ et aux racines simples $\alpha_i$). On note $B^-$ le sous-groupe de Borel oppos\'e et $U$ (resp. $U^-$) le radical unipotent de $B$ (resp. $B^-$); on a $B=T\ltimes U$, $B^-=T\ltimes U^-$ et $U\cap U^-=\{1\}$. On peut consid\'erer les points de $G,T,B,B^-,...$ dans toute $\mathbb C-$alg\`ebre. En particulier $G(\mathscr K)$ et $G(\mathcal O)$ sont des ind-sch\'emas en groupes sur $\mathbb C$.

La grassmannienne affine est $\mathcal G=G(\mathscr K)/G(\mathcal O)$, c'est une ind-vari\'et\'e sur $\mathbb C$.

On suppose en fait, dans cette section, $G$ simplement connexe, donc $Y=Q^\vee$ et $T=Y\otimes_{\mathbb Z}\mathbb C^*$. Pour $\lambda\in Y$, on note $t^\lambda=t\otimes\lambda\in Y\otimes\mathscr K=T(\mathscr K)\subset G(\mathscr K)$. Ainsi $Q^\vee$ s'identifie \`a un sous-groupe de $T(\mathscr K)$ et on a $T(\mathscr K)=Q^\vee\ltimes T(\mathcal O)$, donc $Q^\vee\cap G(\mathcal O)=\{1\}$ : $Q^\vee$ se plonge dans $\mathcal G$.

On consid\`ere l'orbite $\mathcal G_\lambda=G(\mathcal O).t^\lambda$ de $t^\lambda$ sous $G(\mathcal O)$ dans $\mathcal G$; comme $ G(\mathcal O)$ contient $N$ qui agit sur $Q^\vee$ comme $W^v$, on peut supposer, sans changer $\mathcal G_\lambda$, que $\lambda\in Q^{\vee+}=Y^+$. La d\'ecomposition de Bruhat $G(\mathscr K)=G(\mathcal O)T(\mathscr K)G(\mathcal O)$ montre que $\mathcal G$ est r\'eunion disjointe des $\mathcal G_\lambda$ pour $\lambda\in Y^+$. L'adh\'erence $\overline{\mathcal G}_\lambda$ de $\mathcal G_\lambda$ dans $\mathcal G$ est en fait une vari\'et\'e alg\'ebrique complexe projective de dimension $(2\lambda,\rho)$ o\`u $\rho$ est la demi-somme des racines positives de $\Phi$ ($\rho\in X$, dual de $Y$).

Pour $\mu\in Y$, on peut aussi consid\'erer l'orbite $S^-_\mu=U^-(\mathscr K).t^\mu$ de $t^\mu$ sous $U^-(\mathscr K)$ dans $\mathcal G$. La d\'ecomposition d'Iwasawa $G(\mathscr K)=U^-(\mathscr K)T(\mathscr K)G(\mathcal O)$ nous dit que $\mathcal G$ est r\'eunion disjointe des $S^-_\mu$ pour $\mu\in Y$. On note $S_{\lambda,\mu}=S^-_\mu\cap\mathcal G_\lambda\subset\mathcal G$, c'est une vari\'et\'e alg\'ebrique complexe. Mirkovi\'c et Vilonen ont montr\'e (voir \cite{MV}) que toutes les composantes irr\'eductibles de $S_{\lambda,\mu}$ sont de dimension $(\lambda+\mu,\rho)$.

Les cycles de Mirkovi\'c-Vilonen sont les composantes irr\'eductibles de l'adh\'erence $\overline S_{\lambda,\mu}$ de $S_{\lambda,\mu}$ dans $\mathcal G$, ils ont donc pour dimension $(\lambda+\mu,\rho)$.

Ces cycles fournissent une autre formule des caract\`eres: la multiplicit\'e de $\mu$ dans la repr\'esentation $V(\lambda)$ de $G^\vee$ est le nombre de composantes irr\'eductibles de $\overline S_{\lambda,\mu}$.

On va faire le lien avec les chemins LS et la formule des caract\`eres de Littelmann (Th\'eor\`eme \ref{thCaract}). Pour cela on va utiliser l'immeuble de Bruhat-Tits $\mathcal I$ de $G$ sur $\mathscr K$ et on commence par rappeler quelques r\'esultats sur les immeubles affines.

\subsection{Immeubles affines}

Un immeuble affine (ou vectoriel) est un espace m\'etrique $\mathcal I$ recouvert par une famille de sous-espaces appel\'es appartements, tous isom\'etriques \`a l'appartement t\'emoin $\mathbb A$ du paragraphe \ref{sseAppart} par une isom\'etrie unique \`a $W^a$ pr\`es. Les alc\^oves de l'immeuble sont les alc\^oves de ses appartements. L'axiome fondamental suivant est satisfait:

Deux alc\^oves de l'immeuble sont contenues dans un m\^eme appartement et cet appartement est unique \`a un isomorphisme fixant (point par point) les deux alc\^oves pr\`es.

Si $\mathfrak a$ est une alc\^ove dans un appartement $A$ de $\mathcal I$, pour tout $x\in\mathcal I$, il existe un appartement $B$ contenant $\mathfrak a$ et $x$ et un isomorphisme $\varphi$ de $B$ sur $A$ fixant $\mathfrak a$. L'\'el\'ement $\varphi(x)\in A$ ne d\'epend pas des choix de $B$ et $\varphi$, on le note $\rho_{A,\mathfrak a}(x)$. L'application $\rho_{A,\mathfrak a}:\mathcal I\to A$ est la r\'etraction de $\mathcal I$ sur $A$ de centre $\mathfrak a$; elle diminue les distances, conserve les types et transforme une galerie (d'alc\^oves) en une autre galerie.

L'enclos $cl(\Omega)$ d'une partie $\Omega$ d'un appartement $A$ est l'intersection des demi-appartements de $A$ contenant $\Omega$. L'intersection de deux appartements $A,B$ est close (i.e. \'egale \`a son enclos) et les deux appartements sont isomorphes par un isomorphisme fixant leur intersection. Une galerie minimale dans un appartement reste dans l'enclos de ses extr\'emit\'es; ainsi dans l'immeuble une galerie minimale est dans tout appartement contenant ses extr\'emit\'es.

Un demi-appartement $D$ de mur $M=\partial D$ et une alc\^ove $\mathfrak a$ dont une cloison est dans $M$ sont toujours contenus dans un m\^eme appartement.

L'immeuble $\mathcal I$ est dit \'epais si toute cloison est contenue dans au moins 3 chambres. C'est v\'erifi\'e pour l'immeuble de $G$ sur $\mathscr K$, plus pr\'ecis\'ement dans ce cas, l'ensemble des alc\^oves contenant une cloison donn\'ee d'une alc\^ove $\mathfrak a_0$ (mais diff\'erentes de $\mathfrak a_0$) est param\'etr\'e par $\mathbb C$.

Pour plus de d\'etails sur les immeubles affines, on pourra se r\'ef\'erer \`a \cite{AB}, \cite{B} ou \cite{R}.

\subsection{Germes de quartier}

Un quartier dans $\mathbb A$ est un sous-ensemble de la forme $\mathfrak Q=x+C$ pour un point $x\in \mathbb A$ (son sommet) et une chambre de Weyl $C\subset V$ (sa direction). Deux quartiers sont \'equipollents si et seulement si leur intersection contient un autre quartier. Les classes d'\'equivalence sont les germes de quartier, elles sont en bijection avec les chambres de Weyl de $V$.

On a donc une notion de quartier ou de germe de quartier dans l'immeuble $\mathcal I$. Les immeubles de Bruhat-Tits, qui nous int\'eressent, ont deux propri\'et\'es particuli\`eres:

Deux germes de quartier sont toujours contenus dans un m\^eme appartement (contenir un germe de quartier signifie contenir l'un des quartiers du germe). Cela permet de montrer que ces germes constituent l'ensemble des chambres d'un immeuble sph\'erique $\mathcal I^s$ (au sens de \cite{T74}). On reviendra sur $\mathcal I^s$ \`a la section \ref{seGauss} (sous le nom $I_\infty$).

Un germe de quartier et une alc\^ove sont toujours contenus dans un m\^eme appartement (c'est une cons\'equence de la d\'ecomposition d'Iwasawa). Comme dans le paragraphe pr\'ec\'edent cela permet de d\'efinir une r\'etraction $\rho_{A,\mathfrak q}$ de l'immeuble sur un appartement $A$ de centre un germe de quartier $\mathfrak q$ contenu dans $A$. Cette r\'etraction aussi diminue les distances, conserve les types et transforme les galeries (d'alc\^oves) en galeries.

\subsection{L'immeuble $\mathcal I$ de $G$ sur $\mathscr K$}
Le groupe $G(\mathscr K)$ agit sur cet immeuble par des isom\'etries permutant les appartements et tous les isomorphismes d'appartements \'evoqu\'es ci-dessus peuvent \^etre choisis induits par un \'el\'ement de $G(\mathscr K)$, cf. \cite{BT} et \cite{BT2}.

1) Les appartements de $\mathcal I$ sont en bijection $G(\mathscr K)-$\'equivariante avec les tores maximaux d\'eploy\'es de $G$ sur $\mathscr K$. L'appartement fondamental (identifi\'e \`a $\mathbb A$) est associ\'e \`a $T$, il est stable par le normalisateur $N(\mathscr K)$. On suppose $G$ simplement connexe, donc $N(\mathscr K)$ agit sur $\mathbb A$ par des \'el\'ements de $W^a=W^v\ltimes Q^\vee$ (en g\'en\'eral il faut consid\'erer $W^v\ltimes Y$); en fait, pour $\mu\in Q^\vee$, $t^\mu$ agit par la translation de vecteur $\mu\in V=Y\otimes\mathbb R$. Ainsi $G(\mathscr K)$ permute les alc\^oves, quartiers, germes de quartiers, ...

2) Le fixateur du sommet sp\'ecial $0\in\mathbb A$ est $G(\mathcal O)$. Ainsi la grassmannienne affine $\mathcal G=G(\mathscr K)/G(\mathcal O)$ s'identifie \`a l'orbite de $0$ dans $\mathcal I$, c'est l'ensemble des sommets (sp\'eciaux) de type $0$ de $\mathcal I$.

L'ensemble des sommets (sp\'eciaux) de type $0$ de $\mathbb A$ est $Q^\vee=Y$, ses orbites sous $N(\mathcal O)$ ont pour syst\`eme de repr\'esentants $Y^+$. La vari\'et\'e $\mathcal G_\lambda$ est l'orbite dans $\mathcal I$ sous $G(\mathcal O)$ du sommet $\lambda\in Y^+$.

3) Si $\mathfrak q$ est un germe de quartier, son stabilisateur $G(\mathscr K)_{\mathfrak q}$, ensemble des \'el\'ements de $G(\mathscr K)$ transformant un quartier de $\mathfrak q$ en un (autre) quartier de $\mathfrak q$, est un sous-groupe de Borel de $G(\mathscr K)$. En fait $\mathcal I^s$ est l'immeuble de Tits de $G$ sur $\mathscr K$ dont les facettes correspondent aux paraboliques.

Le germe de quartier $\mathfrak q(C^v)$, de direction $C^v$ dans $\mathbb A$, a pour stabilisateur $B(\mathscr K)$. Son fixateur, ensemble des \'el\'ements fixant (point par point) un quartier de ce germe, est $U(\mathscr K)T(\mathcal O)$. Le groupe $T(\mathcal O)$ est le fixateur de $\mathbb A$.

De m\^eme le germe de quartier $\mathfrak q(-C^v)$ de direction $-C^v$ a pour stabilisateur $B^-(\mathscr K)$ et pour fixateur $U^-(\mathscr K)T(\mathcal O)$. On consid\`ere la r\'etraction $\rho_{-\infty}=\rho_{\mathbb A,\mathfrak q(-C^v)}$. Si $x\in\mathbb A$, son image r\'eciproque par $\rho_{-\infty}$ est l'orbite de $x$ sous $U^-(\mathscr K)$.

4) Pour $\lambda\in Y^+$ et $\mu\in Y$, on peut donc r\'einterpr\'eter $S_{\lambda,\mu}$ dans l'immeuble: c'est l'ensemble des sommets de $\mathcal G_\lambda=G(\mathcal O)\lambda$ dont l'image par $\rho_{-\infty}$ est $\mu$.

\bigskip
On va faire le lien avec les chemins LS ou plut\^ot avec des galeries LS que l'on va d\'efinir. Pour simplifier, on suppose $\lambda$ r\'egulier (i.e. $(\lambda,\alpha)\ne0,\forall\alpha\in\Phi$); sinon, il faut, soit consid\'erer des galeries de suites de faces d'alc\^oves, cf. \cite{GL}, Section 4, soit, changer un peu les d\'efinitions de galeries LS, cf \cite{BaGa}, Section 5.2.

\subsection{Galeries associ\'ees \`a un copoids $\lambda$}

Soit $\lambda\in Y^+_{reg}$, on consid\`ere, dans $\mathbb A$ une galerie minimale (d'alc\^oves) $\gamma_\lambda$ joignant $0$ \`a $\lambda$, de longueur $p$ et de type $t(\gamma_\lambda)=(i_1,...,i_p)\in(I\cup\{0\})^p$. On note $\mathfrak a_-$ l'alc\^ove de sommet $0$ dans $-C^v$.

Entre deux sommets de $\mathcal I$, il existe au plus une galerie minimale de type $t(\gamma_\lambda)$ et deux telles galeries de m\^eme origine sont conjugu\'ees par le fixateur de cette origine. Ainsi $\mathcal G_\lambda$ est en bijection avec l'ensemble $\widehat{\mathcal G}_\lambda$ des galeries minimales $\delta$ de $\mathcal I$ d'origine contenant $0$ et de type $t(\gamma_\lambda)$. A une telle galerie $\delta=(\mathfrak a_0,...,\mathfrak a_p)$ on peut rajouter une galerie minimale $(\mathfrak a_{-q},...,\mathfrak a_0)$ de $\mathfrak a_{-q}=\mathfrak a_{-}$ \`a l'origine $\mathfrak a_{0}$ de $\delta$, on obtient ainsi une nouvelle galerie minimale $\delta^*$ d'origine $\mathfrak a_{-}$.

Dans $\mathbb A$  on consid\`ere l'ensemble $\Gamma(\gamma_\lambda)$ des galeries de type $t(\gamma_\lambda)$ et d'origine contenant $0$. A une telle galerie $\gamma=(\mathfrak a'_0,...,\mathfrak a'_p)$ on peut rajouter une galerie minimale $(\mathfrak a'_{-q},...,\mathfrak a'_0)$ de $\mathfrak a'_{-q}=\mathfrak a_{-}$ \`a l'origine $\mathfrak a'_{0}$ de $\gamma$, on obtient ainsi une nouvelle galerie  $\gamma^*$ d'origine $\mathfrak a_{-}$.

La r\'etraction $\rho_{-\infty}$ applique \'evidemment $\widehat{\mathcal G}_\lambda$ dans $\Gamma(\gamma_\lambda)$. Si $\rho_{-\infty}(\delta)=\gamma$, alors $\rho_{-\infty}(\delta^*)=\gamma^*$.

\subsection{Plis et murs porteurs}

La galerie $\gamma=(\mathfrak a'_0,...,\mathfrak a'_p)\in \Gamma(\gamma_\lambda)$ est dite pli\'ee positivement si, pour tout $j\in\{1,...,p\}$ tel que $\mathfrak a'_{j-1}=\mathfrak a'_{j}$, le germe de quartier $\mathfrak q(-C^v)$ est s\'epar\'e de $\mathfrak a'_{j}$ par le mur $M_j$ contenant la cloison de type $i_j$ de $\mathfrak a'_{j}$. On note $\Gamma^+(\gamma_\lambda)$ l'ensemble des galeries pli\'ees positivement de $\Gamma(\gamma_\lambda)$.

Un mur $M$ est porteur pour la galerie $\gamma\in\Gamma^+(\gamma_\lambda)$ \`a l'\'etape $j\in\{-q+1,...,p\}$ si l'on est dans l'un des 3 cas suivants:

1) $j\leq0$ et $M=M_j$ s\'epare $\mathfrak a_{-}$ et $\mathfrak a'_{j-1}$ de $\mathfrak a'_{0}$ et $\mathfrak a'_{j}$ (donc $0\in M$).

2) $j>0$, $\mathfrak a'_{j-1}=\mathfrak a'_{j}$ et $M=M_j$ s\'epare $\mathfrak q(-C^v)$ de $\mathfrak a'_{j}$.

3) $j>0$, $\mathfrak a'_{j-1}\ne\mathfrak a'_{j}$ et $M=M_j$ ($\supset \mathfrak a'_{j-1}\cap\mathfrak a'_{j}$) s\'epare $\mathfrak q(-C^v)$ et $\mathfrak a'_{j-1}$ de $\mathfrak a'_{j}$.

\medskip La dimension de $\gamma$ est le nombre $dim(\gamma)$ de ces paires $(M,j)$.

\begin{prop}\label{prDim}
$\rho_{-\infty}(\widehat{\mathcal G}_\lambda)=\Gamma^+(\gamma_\lambda)$ et, pour $\gamma\in\Gamma^+(\gamma_\lambda)$, $dim(\gamma)$ est la <<dimension>> de $(\rho_{-\infty})^{-1}(\gamma)$.
\end{prop}
\begin{proof}[Id\'ee de d\'emonstration] Pour $\gamma\in\Gamma(\gamma_\lambda)$, cherchons $\delta\in\widehat{\mathcal G}_\lambda$ tel que $\rho_{-\infty}(\delta^*)=\gamma^*$. On construit les alc\^oves de $\delta^*$ de proche en proche par r\'ecurrence sur $j\in[-q,...,p]$. Supposons $\mathfrak a_{j-1}$ construit avec $\rho_{-\infty}(\mathfrak a_{j-1})=\mathfrak a'_{j-1}$ et cherchons $\mathfrak a_{j}$, $i_j-$mitoyenne mais diff\'erente de $\mathfrak a_{j-1}$ (i.e. $\mathfrak a_{j-1}\cap\mathfrak a_{j}$ est une cloison de type $i_j$), telle que $\rho_{-\infty}(\mathfrak a_{j})=\mathfrak a'_{j}$.

Si $\mathfrak a'_{j-1}$ et $\mathfrak q(-C^v)$ sont du m\^eme c\^ot\'e de $M_j$ (e.g. si $j\leq0$), alors il en est de m\^eme de $\mathfrak a_{j-1}$ et $\mathfrak q(-C^v)$ dans un appartement les contenant; ainsi toute alc\^ove $\mathfrak b$ $i_j-$mitoyenne mais diff\'erente de $\mathfrak a_{j-1}$ est dans un m\^eme appartement que $\mathfrak a_{j-1}$ et $\mathfrak q(-C^v)$, donc $\rho_{-\infty}(\mathfrak b)\ne\rho_{-\infty}(\mathfrak a_{j-1})$. Ainsi $\rho_{-\infty}(\widehat{\mathcal G}_\lambda)\subset\Gamma^+(\gamma_\lambda)$. De plus, si $\gamma\in\Gamma^+(\gamma_\lambda)$, on peut choisir pour $\mathfrak a_{j}$ tous les $\mathfrak b$ ci-dessus, d'o\`u un ensemble $\mathbb C$ de param\`etres. On est dans le cas 1) ou 3) de la d\'efinition de mur porteur.

Si $\mathfrak a'_{j-1}$ et $\mathfrak q(-C^v)$ sont de part et d'autre de $M_j$, il en est de m\^eme de $\mathfrak a_{j-1}$ et $\mathfrak q(-C^v)$. Parmi les cloisons $\mathfrak b$ $i_j-$mitoyennes mais diff\'erentes de $\mathfrak a_{j-1}$, il y en a une $\mathfrak b_0$ dans le m\^eme appartement que $\mathfrak a_{j-1}$ et $\mathfrak q(-C^v)$, son image $\rho_{-\infty}(\mathfrak b_0)$ est diff\'erente de  $\mathfrak a'_{j-1}$. Pour toutes les autres, $\mathfrak b$, $\mathfrak b_0$ et $\mathfrak q(-C^v)$ sont dans un m\^eme appartement, donc $\rho_{-\infty}(\mathfrak b)\ne\rho_{-\infty}(\mathfrak b_{0})$ et ainsi $\rho_{-\infty}(\mathfrak b)=\mathfrak a'_{j-1}$. Si le mur $M_j$ n'est pas porteur, $\mathfrak b_0$ est le seul choix possible pour $\mathfrak a_{j}$. Si le mur $M_j$ est porteur (cas 2) de la d\'efinition) les choix possibles pour $\mathfrak a_{j}$ sont les $\mathfrak b\ne\mathfrak b_0$; on a donc un ensemble de param\`etres \'egal \`a $\mathbb C^*$.

Comme le type de $\delta^*$ est le type d'une galerie minimale et comme les alc\^oves successives sont \`a chaque fois distinctes, il est classique que $\delta^*$ est minimale.
\end{proof}

\subsection{Galeries LS}

Si $\gamma$ est une galerie de $\mathcal I$ de type $t(\gamma_\lambda)$, on appelle but de cette galerie le sommet de la derni\`ere alc\^ove non dans la cloison de type $i_p$. On note $\Gamma^+(\gamma_\lambda,\mu)$ l'ensemble des galeries de $\Gamma^+(\gamma_\lambda)$  de but $\mu$. On montre dans \cite{GL}

\begin{prop}\label{prGalLS}
Si $\gamma\in \Gamma^+(\gamma_\lambda,\mu)$, alors $dim(\gamma)\leq(\lambda+\mu,\rho)$.
\end{prop}

La galerie $\gamma\in \Gamma^+(\gamma_\lambda,\mu)$ est dite LS si $dim(\gamma)=(\lambda+\mu,\rho)$. On note $\Gamma^+_{LS}(\gamma_\lambda,\mu)$ l'ensemble de ces galeries.

Dans \cite{GL} on montre que ces galeries LS s'obtiennent \`a partir de $\gamma_\lambda$ en it\'erant des op\'erateurs $f_\alpha$ analogues \`a ceux de \ref{sseOpera} et on prouve la formule des caract\`eres suivante :

\begin{thm} \label{thCarGal}
La multiplicit\'e du poids $\mu $ dans la repr\'esentation $V(\lambda)$ de $G^\vee$ est le nombre de galeries dans $\Gamma^+_{LS}(\gamma_\lambda,\mu)$.
\end{thm}

\subsection{Lien avec les cycles de Mirkovi\'c-Vilonen}

La vari\'et\'e $S_{\lambda,\mu}\subset\mathcal G_\lambda$ est l'ensemble des buts de galeries de  $\widehat{\mathcal G}_\lambda$ dont l'image par $\rho_{-\infty}$ est $\mu$. Elle est donc en bijection avec $(\rho_{-\infty})^{-1}(\Gamma^+_{LS}(\gamma_\lambda,\mu))$, r\'eunion des $(\rho_{-\infty})^{-1}(\gamma)$ qui sont de dimension $dim(\gamma)\leq(\lambda+\mu,\rho)$. Les composantes irr\'eductibles de $S_{\lambda,\mu}$ \'etant de dimension $(\lambda+\mu,\rho)$, ce sont les $(\rho_{-\infty})^{-1}(\gamma)$ pour $\gamma\in \Gamma^+_{LS}(\gamma_\lambda,\mu)$. D'o\`u le lien entre la formule des caract\`eres du Th\'eor\`eme \ref{thCaract} ci-dessus et celle du paragraphe \ref{sseGrass}.

Dans le raisonnement approximatif pr\'ec\'edent, on a n\'eglig\'e de prendre les adh\'erences dans $\mathcal G$. Soit $\widehat\Sigma(\gamma_\lambda)$ l'ensemble de toutes les galeries (minimales ou non) dans $\mathcal I$ d'origine contenant $0$ et de type $t(\gamma_\lambda)$; c'est une vari\'et\'e alg\'ebrique projective lisse. L'application <<but>> est birationnelle d'image $\overline{\mathcal G}_\lambda$: on a une r\'esolution des singularit\'es de $\overline{\mathcal G}_\lambda$ analogue \`a la vari\'et\'e de Bott-Samelson associ\'ee \`a une vari\'et\'e de Schubert.

Comme $\widehat{\mathcal G}_\lambda$ est un ouvert de $\widehat\Sigma(\gamma_\lambda)$ (d\'efini par le fait que toutes les alc\^oves d'une galerie sont diff\'erentes), ${\mathcal G}_\lambda$ est un ouvert de $\overline{\mathcal G}_\lambda$ et le raisonnement approximatif ci-dessus peut \^etre consolid\'e pour prouver la bijection entre les cycles de Mirkovi\'c-Vilonen de $\overline{S}_{\lambda,\mu}$ et les galeries de $\Gamma^+_{LS}(\gamma_\lambda,\mu)$, voir \cite{GL} et \cite{BaGa}.

\subsection{Lien entre galeries LS et chemins LS}
\label{sseGalChe}

On peut supposer que la galerie $\gamma_\lambda$ contient le segment $[0,\lambda]$. Toute galerie $\gamma\in\Gamma(\gamma_\lambda)$ est obtenue par une succession de pliages \`a partir d'une image de $\gamma_\lambda$ par un $w\in W^v$, elle contient un chemin $\pi_\gamma$ obtenu par pliage du chemin $w\pi_\lambda$ et $\pi_\gamma(1)$ est le but de $\gamma$. Comme les pliages de $\gamma$ se font selon des murs, $\pi_\gamma$ est un vrai billard et il est facile de voir que, si $\gamma$ est pli\'ee positivement, alors $\pi_\gamma$ l'est aussi, on peut m\^eme montrer que $\pi_\gamma$ est Hecke.

Si on appelle mur porteur pour un chemin $\pi$ tout mur quitt\'e positivement par $\pi$ en $t\in[0,1[$ (i.e. $\pi(t)$ est dans le mur et $\pi(t+\varepsilon)$ est du c\^ot\'e positif) alors  la condition {\bf (LS3)} (de d\'efinition des chemins LS) \'equivaut au fait que le nombre de murs porteurs pour $\pi$ est $(\lambda + \pi(1),\rho)$. Ainsi les galeries LS correspondent aux chemins LS. Ce point de vue est d\'evelopp\'e dans \cite{GR}.

Dans ce cadre la Proposition \ref{prDim} devient:

\begin{prop} Un chemin dans $\mathbb A$ est l'image par $\rho_{-\infty}$ d'un segment de $\mathcal I$ (parcouru \`a vitesse constante) si, et seulement si, il est de Hecke.
\end{prop}

La d\'emonstration est tout \`a fait analogue \`a celle du Th\'eor\`eme \ref{thPliage} ci-dessous.

\newpage
\section{Pliage et d\'epliage de triangles}
\label{seDepli}

\subsection{Du local au global}
\subsubsection{Immeuble tangent}
\label{ssseTang}

Dans cette partie, $\mathcal I$ d\'esigne un immeuble affine (ou vectoriel) \'epais g\'en\'eral. Soit $x\in\mathcal I$, l'\'etoile $x^*$ (r\'eunion des facettes qui contiennent $x$) est, comme ensemble de facettes, un immeuble combinatoire sph\'erique \'epais. Sa r\'ealisation vectorielle, not\'ee $\Sigma_x(\mathcal I)$, est l'immeuble tangent de $\mathcal I$ en $x$.

En fait $\Sigma_x(\mathcal I)$ est le quotient de $x^*\times\mathbb R_+$ par la relation: $(y,\lambda)\sim(z,\mu)\Leftrightarrow$ $y$ et $z$ sont dans le m\^eme segment d'origine $x$ d'une facette de $x^*$ et, pour $N$ grand, $(1-\lambda/N)x+(\lambda/N)y=(1-\mu/N)x+(\mu/N)z$. On note $\lambda\overrightarrow{xy}$ la classe de $(y,\lambda)$.
Le point $x$ s'identifie au point $0=0_x=\lambda\overrightarrow{xx}=0\overrightarrow{xy}$ de $\Sigma_x(\mathcal I)$.

 Les appartements de $\Sigma_x(\mathcal I)$ sont les espaces vectoriels
$$
\overrightarrow{Ax} = \{ \lambda \overrightarrow{xy}\mid y\in A\cap x^*\,,\,\lambda\in\mathbb R_+\},
$$
pour $A$ un appartement affine contenant $x$. Les chambres de $\Sigma_x(\mathcal I)$ sont les $\vec{c}=\{ \lambda \overrightarrow{xy}\mid y\in c\,,\,\lambda\in\mathbb R_+\}$ pour $c$ une alc\^ove de $x^*$. Le groupe de Weyl de $\overrightarrow{Ax}$ est  le groupe $W^v_x$ engendr\'e par les r\'eflexions lin\'eaires associ\'ees aux r\'eflexions affines de $A$ fixant $x$.

Si $\vec v\in\Sigma_x(\mathcal I)$, alors $x+\vec v$ est bien d\'efini dans $\mathcal I$ si $\Vert\vec v\Vert$ est assez petit pour que $x+\vec v\in x^*$, sinon, il faut pr\'eciser un appartement $A\ni x$ tel que $\overrightarrow{Ax}\ni \vec v$. Deux \'el\'ements $\vec v$ et $\vec v'$ de  $\Sigma_x(\mathcal I)$ sont dits oppos\'es s'ils sont dans un m\^eme appartement $\overrightarrow{Ax}$ et oppos\'es dans celui-ci. On note $\vec v'=op_{\overrightarrow{Ax}}(\vec v)$.

Si $\pi$ est un chemin lin\'eaire par morceaux (donc r\'eunion de segments contenus dans des appartements), on peut d\'efinir, pour tout $t$,  les d\'eriv\'ees $\pi'_+(t)$ et $-\pi'_-(t)$ dans $\Sigma_{\pi(t)}\mathcal I$ de mani\`ere intrins\`eque (i.e. ind\'ependante du choix d'appartements). Le chemin $\pi$ est d\'erivable en $t$ si et seulement si $\pi'_+(t)$ et $-\pi'_-(t)$ sont oppos\'es dans $\Sigma_{\pi(t)}\mathcal I$

\subsubsection{Diff\'erentielle}

Soit $f:\mathcal I\to\mathcal J$ un morphisme d'immeubles affines. En particulier, $f(x^*) \subset (f(x))^*$. Pour tout $x\in \mathcal I$, on d\'efinit la diff\'erentielle de $f$ comme \'etant la r\'ealisation vectorielle de la restriction de $f$ \`a $x^*$ et on la note $df_x : \Sigma_x(\mathcal I) \to \Sigma_{f(x)}(\mathcal J)$. C'est un morphisme d'immeubles vectoriels.

Si $\vec v\in\Sigma_x(\mathcal I)$, soit $\varepsilon >0$ tel que $x+\varepsilon\vec v\in x^*$ alors $f(x+\varepsilon\vec v)\in f(x)^*$. On a
$$
df_x(\vec v) = \frac{1}{\varepsilon}\overrightarrow{f(x)f(x+\varepsilon\vec v)}\ .
$$
\noindent{\bf Propri\'et\'e. } Soient $A$ un appartement, $c$ une alc\^ove et $\rho = \rho_{A,c}$ la r\'etraction sur $A$ centr\'ee en $c$. Si $x\in A$ alors $d\rho_x = \rho_{\overrightarrow{Ax},\vec c_x}$, o\`u $\vec c_x$ est la chambre de $\overrightarrow{Ax}$ qui contient tous les $\overrightarrow{xy}$, $y\in c$ (autrement dit $\vec c_x = \overrightarrow{proj_x(c)}$, si $proj_x(c)$ est l'alc\^ove de $x^*$ \`a distance minimale de $c$).

Si $x\not\in A$, soit $B\supset c\cup \{x\}$ un autre appartement, soit $\varphi\; :$
$
\begin{CD}
B@>\simeq>{}>A
\end{CD}
$ l'isomorphisme qui fixe $c$ et $x$, alors $d\rho_x = d\varphi_x\circ \rho_{\overrightarrow{Bx},\vec c_x}$. En effet $\rho=\varphi\circ\rho_{B,c}$.

\subsubsection{Crit\`ere local}

\begin{prop}
\label{prLocGlob}
Soit $[z,x,x_1,...,x_n,y,z]$ un polygone dans un appartement $A$. Soit $\pi:[0,1]\to A$  une param\'etrisation (lin\'eaire par morceaux et \`a vitesse constante) de $[x,x_1,...,x_n,y]$ avec $x_i = \pi(t_i)$. Soit $a$ une alc\^ove contenant $z$ et soit $\rho = \rho_{A,a}$ la r\'etraction sur $A$ centr\'ee en $a$.

Alors $[z,x,x_1,...,x_n,y,z]$ est l'image par $\rho$ d'un triangle $[z,x,\tilde y,z]$ si, et seulement si, pour tout $i\in\{1,...,n\}$, $[\overrightarrow{x_iz}, -\pi'_-(t_i), 0_{x_i}, \pi'_+(t_i), \overrightarrow{x_iz}]$ est l'image d'un triangle $[\overrightarrow{x_iz}, -\pi'_-(t_i),  \eta_i, \overrightarrow{x_iz}]$ dans $\Sigma_{x_i}(\mathcal I)$ par $\rho_{\overrightarrow{Ax_i}, \vec a_{x_i}}$,  o\`u $\vec a_{x_i}$ est la chambre de $\overrightarrow{Ax_i}$ qui contient tous les vecteurs $\overrightarrow{x_iz'}$, $z'\in a$.

\end{prop}

\begin{rem*}
La condition est \'equivalente au fait que pour tout $i$, il existe $\eta_i\in\Sigma_{x_i}(\mathcal I)$ tel que $\eta_i$ est oppos\'e \`a $-\pi'_- = -\pi'_-(t_i)$ et $\rho(\eta_i)=\pi'_+(t_i)$.
\end{rem*}

\begin{proof} On suppose pour simplifier qu'il existe un groupe $\mathbf G$ agissant sur $\mathcal I$ fortement transitivement (c'est \`a dire transitivement sur les paires form\'ees d'une alc\^ove dans un appartement). Cela nous suffit, car nous appliquerons tout ceci \`a la situation o\`u $\mathcal I = \mathcal I(G,\mathscr K)$, l'immeuble de Bruhat-Tits associ\'e au groupe $G$ sur le corps $\mathscr K$. Apr\`es avoir d\'emontr\'e la proposition \ref{prPliDep}, on peut voir que cette hypoth\`ese est inutile.

\medskip
\noindent{\it Pliage. } On suppose que $[z,x,x_1,...,x_n,y,z]$ est l'image par $\rho$ d'un triangle $[z,x,\tilde y,z]$. Soit $\tilde\pi : [0,1]\to \mathcal I$ une param\'etrisation de $[x,\tilde y]$ telle que $\rho \circ\tilde\pi = \pi$.  Pour tout $i$, posons $\tilde x_i = \tilde\pi(t_i)$ de sorte que $x_i = \rho (\tilde x_i)$. On d\'erive en $t_i$ :
$$
d\rho_{\tilde x_i} (\pm \tilde\pi'_\pm (t_i)) = \pm \pi'_\pm(t_i),\quad d\rho_{\tilde x_i}(\overrightarrow{\tilde x_iz}) = \overrightarrow{x_iz}.
$$
Soit $Z$ un appartement contenant $a, \tilde x_i$ et tel que $-\tilde\pi'_-(t_i)\in\overrightarrow{Z\tilde x_i}$. Soit $g\in\mathbf G$ tel que $g\cdot Z = A$ et $g\cdot a = a$. Alors, $\rho\vert_{ Z} = g\vert_{Z}$. On prend $\eta_i = dg_{\tilde x_i}(\tilde\pi'_+(t_i))$ ; $dg_{\tilde x_i}$ est un isomorphisme entre $\Sigma_{\tilde x_i}\mathcal I$ et $\Sigma_{x_i}\mathcal I$, d'o\`u $\eta_i$ est oppos\'e \`a $-\pi'_-(t_i)$. Il reste \`a montrer que $\rho_{\overrightarrow{Ax_i}, \vec a_{x_i}} (\eta_i) = \pi'_+(t_i)$.

Soit $Y$ un appartement contenant $a,\tilde x_i$ et tel que $\tilde\pi'_+(t_i)\in\overrightarrow{Y\tilde x_i}$. On a des isomorphismes:
$$
\begin{CD}
Y @>g\cdot>a> gY @>\varphi, \ \simeq>a> A\ ,
\end{CD}
$$
\qquad\qquad(cette notation signifie que $g$ et $\varphi$ fixent $a$)

L'isomorphisme $\varphi$ est \'egal \`a $\rho\vert_{ gY}$ et la compos\'ee  \`a $\rho\vert_{ Y}$. On d\'erive :
$$
\begin{CD}
\overrightarrow{Y\tilde x_i} @>dg>> \overrightarrow{gY\tilde x_i} @>d\varphi>> \overrightarrow{A x_i}\ ,
\end{CD}
$$

la compos\'ee est \'egale \`a $d\rho_{\tilde x_i}$ et $\tilde\pi'_+(t_i)\mapsto \eta_i\mapsto d\rho_{\tilde x_i}(\tilde\pi'_+(t_i)) = \pi'_+(t_i)$. Or $d\varphi_{\tilde x_i} : \overrightarrow{gY\tilde x_i} \longrightarrow \overrightarrow{Ax_i}$ est aussi \'egal \`a $\rho_{\overrightarrow{Ax_i}, \vec a_{x_i}}$, donc c'est gagn\'e !

\medskip
\noindent{\it D\'epliage. }  On suppose qu'il existe, pour tout $i$, $\eta_i\in\Sigma_{x_i}\mathcal I$ satisfaisant aux conditions de l'\'enonc\'e. Soit $B_1$ un appartement contenant $a\cup\{x_1\}$ tel que $\eta_1\in \overrightarrow{B_1x_1}$. On pose $\tilde x_2 = x_1 +\lambda\eta_1\in B_1$, o\`u $\lambda$ est tel que $\lambda\Vert \eta_1\Vert = \Vert \overrightarrow{x_1x_2}\Vert$. Notons $\varphi_1$ l'isomorphisme $\begin{CD} B_1@>\simeq>a> A\end{CD}$. Alors $(d\varphi_1)_{x_1}(\eta_1) = \pi'_+(t_1)$ car $(d\varphi_1)_{x_1} = (\rho_{\overrightarrow{A}, \vec a})\vert_{\overrightarrow{B_1}}$. Donc $\varphi_1(x_1 + \lambda\eta_1) = x_1 + d\varphi_1(\lambda\eta_1) = x_1 + \lambda\pi'_+(t_1) = x_2$. Et $\overrightarrow{x_1x_0}$ est oppos\'e \`a $\eta_1$ dans $\Sigma_{x_1}\mathcal I$, donc $[x,x_1,\tilde x_2]$ est un segment dans $\mathcal I$.

Soit $g_1\in\mathbf G$ tel que $g_1\cdot A = B_1$ et $g_1\cdot a =a$. On pose $\tilde x_3 = \tilde x_2 +\lambda_2dg_1(\eta_2)$ dans un appartement $B_2$... On construit de la sorte $\tilde x_3,...,\tilde x_n, \tilde y$ pour fabriquer un triangle.
\end{proof}

\subsection{Pliage}\label{ssePliage}

Soit $\mathcal S$ un immeuble sph\'erique de groupe de Weyl $W^v$; on le consid\`ere dans sa r\'ealisation vectorielle, comme $\Sigma_x(\mathcal I)$ en \ref{ssseTang}. Soient $\xi\in\mathcal S$, $\eta$ oppos\'e \`a $\xi$, $-C$ une chambre et $A$ un appartement contenant $\xi$ et $-C$. On note $\rho = \rho_{A,-C}$ la r\'etraction sur $A$ centr\'ee en $-C$. Le but de cette partie est de trouver une relation entre $\xi$, $-C$ et $\rho(\eta)$ (HR(0) ci-dessous).

\medskip
\noindent{\bf Rappel. } La notion de $W^v-$cha\^{\i}ne (cf. \ref{sseCheKM}) d\'epend du choix de la chambre de Weyl $C^v$. On la reprend ici en explicitant la chambre dans le nom:

Une $(W^v,-C)-$cha\^{\i}ne ou  $(-C)-$cha\^{\i}ne de $op_A(\xi)$ \`a $\rho(\eta)$ est une suite
$$
op_A(\xi), \tau_1 op_A(\xi), \tau_2\tau_1 op_A(\xi),...., \rho(\eta) = \tau_n\cdots \tau_2\tau_1 op_A(\xi),
$$ o\`u chaque $\tau_i$ est une r\'eflexion de $A$ qui \'eloigne $\tau_{i-1}\cdots \tau_2\tau_1 op_A(\xi)$ de $-C$, c'est-\`a-dire :
$$
\tau_{i-1}\cdots \tau_2\tau_1 op_A(\xi),\ -C\quad \mid_{M_i} \quad\tau_i\tau_{i-1}\cdots \tau_2\tau_1 op_A(\xi),
$$ o\`u $M_i$ est le mur associ\'e \`a la r\'eflexion $\tau_i$.

(Cette notation signifie que les termes de gauche sont d'un c\^ot\'e du mur $M_i$ et le terme de droite de l'autre c\^ot\'e.)

\begin{dfn}
Soit $\Gamma$ une galerie minimale de $-C$ \`a $\rho(\eta)$,  $\Gamma = (-C = C_0, C_1,...,C_n\ni \rho(\eta))$. Soit $\tau_i$ la r\'eflexion qui \'echange $C_{i-1}$ et $C_i$. Une cha\^{\i}ne de longueur $\ell$ le long de $\Gamma$ de $op_A(\xi)$ \`a $\rho(\eta)$ est une $(-C)-$cha\^{\i}ne de $op_A(\xi)$ \`a $\rho(\eta)$ du type $\rho(\eta) = \tau_{i_\ell}\cdots \tau_{i_1} op_A(\xi)$ avec $i_1<\cdots <i_\ell$.
\end{dfn}

Soit maintenant une galerie minimale $\Gamma = (C_0 = -C, C_1,..., C_n)$ de $-C$ \`a $\eta$. Pour tout $i$, soit $B_i$ un appartement contenant $\xi$ et $C_i$. Notons $\rho_i  = \rho_{B_i, C_i}$ la r\'etraction centr\'ee en $C_i$ sur $B_i$. Bien \'evidemment, on prend $B_0 = A$ et $\rho_0 = \rho$. De l'autre c\^ot\'e, $B_n$ contient $\xi$ et $\eta$, et $\eta = op_{B_n} \xi$.

\begin{lem}
\label{leRetrac}
Pour $i<j$, on a $(\rho_i)\vert_{ \Gamma_{\geqslant j}} = \rho_i \circ (\rho_j)\vert_{ \Gamma_{\geqslant j}}$.
\end{lem}

\begin{proof} La galerie $\theta = (C_0,...,C_j=\rho_j(C_j),\rho_{j}(C_{j+1}),...,\rho_j(C_n))$ est tendue. En effet, $\rho_j(\theta) = \rho_j(\Gamma)$ est tendue car $\rho_j$ est une r\'etraction centr\'ee en une chambre de $\Gamma$. Ainsi, $\theta$ est tendue car $\rho_j$ r\'eduit les distances.

Soit $Y$ un appartement contenant $\theta$. De m\^eme, soit $Y'$ un appartement contenant $(C_0,...,C_i =\rho_i(C_i),\rho_{i+1}(C_{i+1}),...,\rho_i(C_n))$. Et enfin, soit $Z$ un appartement contenant $\Gamma$. On a :
$$
\begin{CD}
Z @>\varphi_j,\simeq> \Gamma_{\leqslant j}> Y @>\varphi_i,\simeq> \Gamma_{\leqslant i}> Y' \ .
\end{CD}
$$ Ce qui donne $\varphi_j(\Gamma_{\geqslant j}) = \rho_j  (\Gamma_{\geqslant j})$ et $(\varphi_j)\vert_{\Gamma_{\geqslant j}} = (\rho_j)\vert_{\Gamma_{\geqslant j}}$. De m\^eme, $\varphi_i\circ (\varphi_j)\vert_{\Gamma_{\geqslant i}} = (\rho_i)\vert_{\Gamma_{\geqslant i}}$. Et enfin, $(\varphi_i)\vert_{\theta_{\geqslant i}} = (\rho_i)\vert_{\theta_{\geqslant i}}$. D'o\`u, $(\rho_i)\vert_{ \Gamma_{\geqslant j}} = \rho_i \circ (\rho_j)\vert_{ \Gamma_{\geqslant j}}$.
\end{proof}

Dans les conditions de ce num\'ero \ref{ssePliage}, on veut montrer les r\'esultats suivants:

\bigskip
\noindent{\bf HR(i). } Il existe une cha\^{\i}ne le long de $\rho_i( \Gamma_{\geqslant i})$ (c'est donc une $C_i-$cha\^{\i}ne) de $op_{B_i}\xi$ \`a $\rho_i(\eta)$.

\medskip
L'hypoth\`ese HR(n) est trivialement vraie. On suppose HR(i+1), montrons HR(i). Il y a deux cas.

\medskip
\noindent{\it Cas 1. } Il existe un appartement $Y$ contenant $C_i$, $\xi$ et le demi-appartement  $D_{B_{i+1}}(m_i,C_{i+1})$ de  $B_{i+1}$ contenant $C_{i+1}$ et dont le mur $M_i$ contient la cloison $m_i = C_i\cap C_{i+1}$. (C'est le cas si $\xi\in  D_{B_{i+1}}(m_i,C_{i+1})$ ou si $C_i\in B_{i+1}$ et dans ce cas $Y=B_{i+1}$ convient.)

Comme $(C_i, C_{i+1},\rho_{i+1}(C_{i+2}),...,\rho_{i+1}(C_{n}))$ est une galerie tendue, elle ne coupe le mur $M_i$  qu'une seule fois. Donc, $(C_{i+1},\rho_{i+1}(C_{i+2}),...,\rho_{i+1}(C_{n}))\subset D_{B_{i+1}}(m_i,C_{i+1})$. Et il existe des isomorphismes
$$
\begin{CD}
B_{i+1} @>\simeq> \xi,\rho_{i+1}(\Gamma_{\geqslant i+1})> Y @>(\rho_i)_{\mid Y},\simeq>\xi, C_i> B_i \ ;
\end{CD}
$$ par le Lemme \ref{leRetrac}, $\rho_i $ transforme $\rho_{i+1}(\Gamma_{\geqslant i+1})$ en $ \rho_{i}(\Gamma_{\geqslant i+1})$. On a la cha\^{\i}ne suivante dans $B_{i+1}$ le long de $\rho_{i+1}(\Gamma_{\geqslant i+1})$ :
$$
\rho_{i+1} (\eta) =\tau_{i_k}\cdots \tau_{i_1} op_{B_{i+1}} \xi \ .
$$ Par les deux isomorphismes pr\'ec\'edents, on a
$$
\rho_{i} (\eta) =\tau'_{i_k}\cdots \tau'_{i_1} op_{B_{i}} \xi \ ,
$$ o\`u $\tau'_{i_j}$ est la r\'eflexion dans $B_i$ selon la $i_j-$i\`eme cloison de $\rho_i(\Gamma)$. C'est bien une cha\^{\i}ne de $op_{B_i} \xi$ \`a $\rho_i(\eta)$ le long de $\rho_i(\Gamma_{\geqslant i})$.

\medskip
\noindent{\it Cas 2. } La chambre $C_i$ n'est pas dans $B_{i+1}$ et, dans $B_{i+1}$, on a:
$$
\xi\quad \mid_{M_i} \quad (C_{i+1}, \rho_{i+1}(C_{i+2}),..., \rho_{i+1}(C_{n}))\ .
$$ L'intersection $B_i\cap B_{i+1}$ contient $\xi$ et $m_i$ et donc l'enclos $Cl(\xi,m_i)$. Or $\xi\not\in M_i$ donc $Cl(\xi,m_i)$ est de dimension maximale ; c'est l'adh\'erence de la r\'eunion des galeries minimales de $\xi$ \`a $m_i$.

Soit $d=proj_{m_i}(\xi)$, c'est une chambre, elle est adjacente \`a $C_i$ dans $B_i$ et \`a $C_{i+1}$ dans $B_{i+1}$. Autrement dit, $d = \sigma_i (C_{i+1}) = \rho_i(C_{i+1})$ avec $\sigma_i$ la r\'eflexion selon $M_i$ dans $B_{i+1}$. Soit $Y$ un appartement contenant $D_{B_{i+1}}(m_i, C_{i+1})\cup C_i$. On a les isomorphisme suivants :
$$
\begin{CD}
B_{i+1} @>\simeq> \rho_{i+1}(\Gamma_{\geqslant i+1})> Y @>(\rho_i)\vert_{Y},\simeq>C_i> B_i \ ;
\end{CD}
$$ on note $\varphi$ la compos\'ee, elle envoie $ \rho_{i+1}(\Gamma_{\geqslant i+1})$ sur $ \rho_{i}(\Gamma_{\geqslant i+1})$. La cha\^{\i}ne dans $B_{i+1}$, $\rho_{i+1}(\eta) = \tau_{i_k}\cdots \tau_{i_1} (op_{B_{i+1}} \xi)$ devient $\rho_{i}(\eta) = \varphi\tau_{i_k}\varphi^{-1}\cdots \varphi\tau_{i_1}\varphi^{-1}\varphi (op_{B_{i+1}} \xi)$. Or $\tau'_{i_k} = \varphi\tau_{i_k}\varphi^{-1}$ est la r\'eflexion selon la $i_k-$i\`eme cloison de $\rho_i(\Gamma)$
et $\varphi (op_{B_{i+1}} \xi) = op_{B_{i}} \varphi(\xi)= op_{B_{i}} \sigma'_i(\xi)$, car $\varphi\vert_{ B_{i+1}\cap B_i} = (\sigma'_i)\vert_{ B_{i+1}\cap B_i}$, o\`u $\sigma'_i$ est la r\'eflexion selon $m_i$ dans $B_{i}$. Ainsi,
$$
\begin{array}{rcl}
\rho_{i}(\eta) & = &  \tau'_{i_k}\cdots \tau'_{i_1} (op_{B_{i}} \sigma'_i(\xi))\\
 & = & \tau'_{i_k}\cdots \tau'_{i_1} \sigma'_i (op_{B_{i}} \xi)\\
\end{array}
$$ est une cha\^{\i}ne le long de $\rho_i(\Gamma_{\geqslant i})$ dans $B_i$.
\qed

\subsection{D\'epliage}\label{sseDepli}

Soit $\mathcal S$ un immeuble sph\'erique \'epais (dans sa r\'ealisation vectorielle et auquel on pense comme l'immeuble tangent \`a $\mathcal I$ en un point $\pi(t)$). Soient $\xi, \pi'_+$ deux points de $\mathcal S$ et $A$ un appartement les contenant, ainsi qu'une chambre, not\'ee $-C$. Comme pr\'ec\'edemment, $\rho$ d\'esigne la r\'etraction sur $A$ de centre $-C$.

On suppose qu'il existe $\Gamma = (C_0 = -C,C_1,...,C_n)$ une galerie minimale de $-C$ \`a $\pi'_+$ et une cha\^{\i}ne le long de $\Gamma$ de $op_A\xi$ \`a $\pi'_+$ de longueur $l$.

\begin{prop}
\label{prDepli}
Il existe $\eta\in\mathcal S$ tel que $\rho(\eta) = \pi'_+$ et $\eta$ est oppos\'e \`a  $\xi$.
\end{prop}

\begin{proof}  On raisonne par r\'ecurrence, la conclusion cherch\'ee est la condition HR($\ell$) ci-dessous.

\medskip
\noindent{\bf HR(i). } Il existe un appartement $B_i$ contenant $\xi$, il existe $\eta_i\in B_i$ tel que $\rho(\eta_i) = \pi'_+$, il existe une galerie minimale $\Gamma_i$ de $-C$ \`a $\eta_i$ qui est dans $B_i$ \`a partir d'un certain rang $k_i$ telle que $\rho(\Gamma_i) = \Gamma$, et il existe une cha\^{\i}ne le long de $(\Gamma_i)_{\geqslant k_i}$ de $op_{B_i} \xi$ \`a $\eta_i$ de longueur $l-i$.

\medskip
HR(0) est vraie pour $B_0 = A$. Si on a HR(i), notons $\Gamma_i = (-C=D_0,...,D_n)$ la galerie minimale et $\eta_i = \tau_{i_u}\cdots \tau_{i_1} op_{B_i} \xi$ la cha\^{\i}ne le long de $(\Gamma_i)_{\geqslant k_i}$ avec $\tau_{i_j}$ la r\'eflexion selon le mur $M_{i_j}$ contenant la $i_j-$i\`eme cloison $m_{i_j}$ de $\Gamma_i$ et $u=\ell-i$. On a
$$
op_{B_i} \xi, D_{k_i},D_{i_1}\quad \mid_{M_{i_1}}\quad \xi,\eta_i, (D_{i_1+1},..., D_n)\ .
$$
Soit $\mathcal D$ un demi-appartement sortant de $B_i$ le long du mur $M_{i_1}$, il existe car $\mathcal S$ est \'epais. On pose $B_{i+1} = \mathcal D\cup D_{B_i}(m_{i_1},\xi)$, $Z = \mathcal D\cup D_{B_i}(m_{i_1}, D_{k_i})$ et on note $\varphi_Z : B_i\to Z$ l'isomorphisme fixant $D_{B_i}(m_{i_1}, D_{k_i})$. On prend
$$
\Gamma_{i+1} = (D_0,...,D_{i_1},\varphi_Z(D_{i_1+1}),..., \varphi_Z(D_n))
$$
et $\eta_{i+1} = \varphi_Z(\eta_i)$. Remarquons que $\Gamma_{i+1} = (D_0,...,D_{k_i-1},\varphi_Z(D_{k_i}),..., \varphi_Z(D_n))$. Par hypoth\`ese de r\'ecurrence $\Gamma_i$ est tendue, du coup $\rho_{B_i,D_{i_1}}(\Gamma_{i+1}) = \rho_{B_i,D_{i_1}}(\Gamma_{i})$ l'est aussi, et donc de m\^eme pour $\Gamma_{i+1}$. Par le Lemme \ref{leRetrac}, on a $\rho\vert_{ \Gamma_{i+1,\geq i_1}} = \rho\circ (\rho_{B_i,D_{i_1}})\vert_ { \Gamma_{i+1,\geq i_1}}$. D'o\`u $\rho({ \Gamma_{i+1,\geq i_1}}) = \rho(\Gamma_{i,\geq i_1}) = \Gamma_{\geq i_1}$ et $\rho(\eta_{i+1}) = \rho(\eta_i) = \pi'_+$.

Par $\varphi_Z$ la cha\^{\i}ne de l'hypoth\`ese de r\'ecurrence devient $\eta_{i+1} = \varphi_Z(\eta_i) = \tau'_{i_u}\cdots \tau'_{i_2} \varphi_Z\tau_{i_1} op_{B_i} \xi$, avec $\tau'_{i_j}$ la r\'eflexion selon la $i_j-$i\`eme cloison de $\Gamma_{i+1}$ dans $Z$. On a $\eta_{i+1} = \tau'_{i_u}\cdots \tau'_{i_1} op_{Z} \varphi_Z(\xi)$. Notons $\psi : Z\to B_{i+1}$ l'isomorphisme qui fixe le demi-appartement $\mathcal D$. En composant par $\psi$, on obtient $\eta_{i+1} = \tau''_{i_u}\cdots \tau''_{i_1} op_{B_{i+1}} \psi\circ\varphi_Z(\xi)$, avec $\tau''_{i_j}=\psi\tau'_{i_j}\psi^{-1}$. Or $op_{B_{i+1}}(\psi\circ\varphi_Z (\xi)) = \tau''_{i_1}op_{B_{i+1}}(\xi)$. Donc $\eta_{i+1} = \tau''_{i_u}\cdots \tau''_{i_2} op_{B_{i+1}}(\xi)$ et c'est encore une cha\^{\i}ne le long de $(\Gamma_{i+1})_{\geq i_1+1}$. En effet, dans $B_i$, pour tout $j$, on avait
$$
\tau_{i_{j-1}}\cdots \tau_{i_1} op_{B_i} (\xi), D_{k_i},...,D_{i_{j}}\quad \mid_{M_{i_j}} \quad \tau_{i_j}\cdots \tau_{i_1} op_{B_i} (\xi), D_{i_{j+1}},...,D_{n}\ .
$$ Par $\psi\circ\varphi_Z$, on a,  pour $j\geq 2$, les positions suivantes dans $B_{i+1}$ (avec $M'_{i_j}=\psi\circ\varphi_Z(M_{i_j})$ ):
$$
\tau''_{i_{j-1}}\cdots \tau''_{i_2} op_{B_{i+1}} (\xi), \varphi_Z(D_{i_1+1}),...,\varphi_Z(D_{i_{j}})\quad \mid_{M'_{i_j}} \quad \tau''_{i_j}\cdots \tau''_{i_2} op_{B_{i+1}} (\xi), \varphi_Z(D_{i_{j+1}}),...,\varphi_Z(D_{n})\ .
$$
\end{proof}

D'apr\`es le paragraphe \ref{ssePliage} et la proposition \ref{prDepli}, on a montr\'e:

\begin{prop}
\label{prPliDep} Dans un appartement $A$ de l'immeuble $\mathcal S$, on consid\`ere des points $\xi$,  $\pi'_+,$ et une chambre $-C$; on note $\rho=\rho_{A,-C}$. Alors il existe dans $\mathcal S$ un point $\eta$ oppos\'e \`a $\xi$ tel que $\rho(\eta) = \pi'_+$ si, et seulement si, il existe une galerie minimale  $\Gamma$ de $-C$ \`a $\pi'_+$ et une cha\^{\i}ne le long de $\Gamma$ de $op_A\xi$ \`a $\pi'_+$.
\end{prop}

\subsection{Galeries pli\'ees positivement}

\medskip
On garde les m\^emes notations qu'en \ref{sseDepli} ci-dessus.

\begin{prop}
\label{prChaines}
Il existe une $-C-$cha\^{\i}ne de $op_A\xi$ \`a $\pi'_+$ si, et seulement si, il existe une cha\^{\i}ne de $op_A\xi$ \`a $\pi'_+$ le long d'une galerie minimale de $-C$ \`a $\pi'_+$.
\end{prop}

\begin{proof} Il suffit de montrer que l'existence d'une $-C-$cha\^{\i}ne de $op_A\xi$ \`a $\pi'_+$ implique celle d'une cha\^{\i}ne de $op_A\xi$ \`a $\pi'_+$ le long d'une galerie minimale de $-C$ \`a $\pi'_+$.

Pour cela, on dira qu'une galerie $\delta = (D_0,D_1,...,D_n)$ de type $(k_0,k_1,...,k_{n-1})$ dans $A$ est pli\'ee positivement par rapport \`a une chambre $D$ si, en notant $M_j$ le mur de $D_j$ de type $k_j$ (commun \`a $D_j$ et $D_{j+1}$), on a
$$
D_j=D_{j+1} \Longrightarrow\qquad D\quad \mid_{M_j} D_j=D_{j+1}\ .
$$

Pour une chambre $D$ et $\lambda\in\mathcal{S}$, on note $w(D,\lambda)$ l'\'el\'ement de $W^v$ de plus petite longueur tel que $\lambda\in w(D,\lambda).D$

Soit $C_\xi = proj_\xi(-C)$ dans $A$. Soit $\Gamma$ une galerie minimale de $-C$ \`a $\pi'_+$. Comme il existe une $-C-$cha\^{\i}ne de $op_A\xi$ \`a $\pi'_+$, $w(-C,op_A\xi) \leqslant w(-C,\pi'_+)$ et donc il existe une galerie $\gamma = (C_0 = -C,C_1,...,C_n)$ de m\^eme type que $\Gamma$ de $-C$ \`a $op_A\xi$. On veut montrer qu'on peut supposer $\gamma$ pli\'ee positivement par rapport \`a $C_\xi$.

Si $\gamma$ ne l'est pas,  soit $j$ le plus petit indice tel qu'on soit dans la situation : \quad$C_\xi, C_j = C_{j+1}\quad \mid_{M_j}$. Alors, comme $\gamma$ aboutit \`a $op_A\xi$, cette galerie va traverser le mur $M_j$ apr\`es l'indice $j$ ou finir dans ce mur. Posons $j_{max} = \max\{k\in]j,n]\mid  M_k = M_j \;{\mathrm ou}\; op_A\xi\in M_k \}$. On d\'efinit une nouvelle galerie $\lambda = (L_0,...,L_n)$ par
$$
L_k =
\left\{
\begin{array}{ll}
 C_k & \hbox{ si } k\leqslant j\\
s_{M_j}C_k & \hbox{ si } j+1\leqslant k\leqslant j_{max}\\
C_k & \hbox{ si } k > j_{max}\ .\\
\end{array}
\right.
$$ Ainsi $\lambda$ devient pli\'ee positivement par rapport \`a $C_\xi$ en $M_j$ et reste de m\^eme type que $\Gamma$. On recommence cette proc\'edure avec $\lambda$ et ainsi de suite... Au final, on obtient une galerie $\delta = (-C=D_0,...,D_n)$ pli\'ee positivement par rapport \`a $C_\xi$ entre $-C$ et $op_A\xi$.

Notons $\{i_1,...,i_r\}\subset \{1,...,n\}$ les indices (ordonn\'es de mani\`ere croissante) o\`u la galerie $\delta$ est pli\'ee. Alors,
$$
\begin{array}{rcl}
\pi'_+ & = & s_{M_{i_1}}\cdots  s_{M_{i_{r-1}}}s_{M_{i_r}} (s_{M_{i_1}}\cdots  s_{M_{i_{r-1}}})^{-1} \cdots (s_{M_{i_1}}s_{M_{i_2}}s_{M_{i_1}}) s_{M_{i_1}} op_A\xi\\
 & = & s_{M_{i_1}}\cdots s_{M_{i_r}} op_A\xi\\
 & = & \tau_{i_r}\cdots \tau_{i_1}op_A\xi\ ,
\end{array}
$$ o\`u $\tau_{i_j} = s_{M_{i_1}}\cdots  s_{M_{i_{j-1}}}s_{M_{i_j}} (s_{M_{i_1}}\cdots  s_{M_{i_{j-1}}})^{-1}$. A chaque \'etape, on s'\'eloigne de $-C$ et on d\'eplie la galerie $\delta$. En effet, apr\`es le premier d\'epliage, on a :
$$
-C, D_1, ..., D_{i_1}\quad \mid_{M_{i_1}}\quad \tau_{i_1}(D_{i_1+1}), \tau_{i_1}(op_A\xi),\xi
$$ car $\delta$ est pli\'ee positivement par rapport \`a $C_\xi$. La galerie
$$
\delta^1 = (-C=D_0,...,D_{i_1}, \tau_{i_1}(D_{i_1+1}),..., \tau_{i_1}(D_{i_2}), \tau_{i_1}(D_{i_2+1}),..., \tau_{i_1}(D_{n}))
$$ est, jusqu'\`a l'indice $i_2$, minimale et donc \'egale \`a $\Gamma_{\leqslant i_2}$. De plus, on sait que
$$
op_A\xi, D_{i_2+1} = D_{i_2}\quad\mid_{M_{i_2}}\quad C_\xi,\xi\ ,
$$ en appliquant $\tau_{i_1}$ on obtient
$$
\tau_{i_1}op_A\xi, \tau_{i_1}D_{i_2+1} = \tau_{i_1}D_{i_2} \quad\mid_{\tau_{i_1}M_{i_2}}\quad \tau_{i_1}\xi\ .
$$ Donc, $-C, \tau_{i_1}op_A\xi, \tau_{i_1}D_{i_2+1} = \tau_{i_1}D_{i_2}$ sont du m\^eme c\^ot\'e de $\tau_{i_1}M_{i_2}$. Ainsi quand on d\'eplie une deuxi\`eme fois par rapport \`a $\tau_{i_1}M_{i_2}$, on s'\'eloigne de $-C$. On a donc obtenu une cha\^{\i}ne de $op_A\xi$ \`a $\pi'_+$ le long de la galerie minimale $\gamma$ de $-C$ \`a $\pi'_+$.
\end{proof}

\subsection{Conclusion}
\label{sseConcl}

D'apr\`es les propositions \ref{prLocGlob}, \ref{prPliDep} et \ref{prChaines}, on a:

\begin{thm}\label{thPliage}
Soient $[z,x,x_1,...,x_n,y,z]$ un polygone dans un appartement $A$ de $\mathcal I$, $\pi:[0,1]\to A$ une param\'etrisation (lin\'eaire par morceaux et \`a vitesse constante) de $[x,x_1,...,x_n,y]$ avec $x_i=\pi(t_i)$, $\mathfrak a$ une alc\^ove contenant $z$ et $\rho=\rho_{A,\mathfrak a}$ la r\'etraction sur $A$ centr\'ee en $\mathfrak a$.

Pour tout $i\in\{1,...,n\}$, on note $\vec{a}_{x_i}$ la chambre de $\overrightarrow{Ax_i}$ qui contient tous les vecteurs $\overrightarrow{x_iz'}$ pour $z'\in\mathfrak a$.

Alors $[z,x,x_1,...,x_n,y,z]$ est l'image par $\rho$ d'un triangle $[z,x,\tilde y,z]$ de $\mathcal I$ si et seulement si, pour  tout $i\in\{1,...,n\}$, il existe une $(W^v_{\pi(t_i)},\vec{a}_{x_i})-$cha\^{\i}ne de $\pi'_-(t_i)$ \`a $\pi'_+(t_i)$.
\end{thm}

On notera la diff\'erence avec un chemin de Hecke (section \ref{sseCheKM}). On dira qu'un chemin $\pi$ v\'erifiant la condition ci-dessus est de Hecke par rapport \`a l'alc\^ove $\mathfrak a$.

\begin{rems}\label{remHecke}

1) Soient $x$ un sommet sp\'ecial de l'alc\^ove $\mathfrak a$ et $Q$ le quartier de sommet $x$ oppos\'e \`a $\mathfrak a$. Un chemin $\pi$ enti\`erement contenu dans $Q$ est de Hecke (par rapport \`a $-C^v$ o\`u $C^v$ est la direction de $Q$) si et seulement si il est de Hecke par rapport \`a $\mathfrak a$.

2) Les th\'eor\`emes \ref{thDecomp} et \ref{thPliage} ainsi que la remarque pr\'ec\'edente prouvent l'\'etape 1) du sch\'ema de d\'emonstration du th\'eor\`eme \ref{thSat}. En effet le chemin de Hecke $N\lambda+\pi$ de $N\lambda$ \`a $N\nu^*$ reste dans la chambre de Weyl $C^v$; il est donc de Hecke par rapport \`a l'alc\^ove $\mathfrak a_-$ (contenant $0$ et oppos\'ee \`a $C^v$). En d\'epliant le polygone $[0,N\lambda,\pi,N\nu^*,0]$ on obtient le triangle cherch\'e dans $\mathcal I$.
\end{rems}

\begin{coro}
\label{corSatTrip}
L'ensemble $\mathcal T$ des triplets $(\lambda,\mu,\nu)\in(P^{\vee+})^3$ tels qu'il existe dans $\mathcal I$ un triangle $[z,x,y,z]$ avec comme longueurs de c\^ot\'es $\lambda=d_{C^v}(z,x)$, $\mu=d_{C^v}(x,y)$ et $\nu=d_{C^v}(y,z)$ ne d\'epend que de l'appartement $\mathbb A$ et non de $\mathcal I$.

Cet ensemble  $\mathcal T(\mathbb A)$ est stable par homoth\'etie par $\mathbb N$.

Si $\mathbb A'$ est un appartement affine (ou vectoriel) associ\'e au m\^eme couple $(V,W^v)$, mais avec un ensemble de murs $\mathcal M'\subset\mathcal M$, alors $\mathcal T(\mathbb A')\subset\mathcal T(\mathbb A)$.
\end{coro}
\begin{proof}
$\mathcal T$ est l'ensemble des $(\lambda,\mu,\nu)\in(P^{\vee+})^3$ tels qu'il existe dans $\mathbb A$ un bon polygone (i.e. v\'erifiant la condition du th\'eor\`eme) $[z,x,x_1,...,x_n,y,z]$ avec $\lambda=d_{C^v}(z,x)$,  $\nu=d_{C^v}(y,z)$ et $\pi$ de type $\mu$. Il ne d\'epend donc que de $\mathbb A$ et contient $\mathcal T(\mathbb A')$. Si $z_0$ est un sommet sp\'ecial d'une alc\^ove contenant $z$, une homoth\'etie de rapport $n\in \mathbb N$ et centre $z_0$ transforme un bon polygone associ\'e \`a $(\lambda,\mu,\nu)$ en un bon polygone associ\'e \`a $(n\lambda,n\mu,n\nu)$; d'o\`u la seconde assertion.
\end{proof}

\noindent{\bf N.B.} On verra dans la section \ref{seGauss} (\ref{prSat}) que le c\^one $\mathcal T$ est satur\'e dans $(P^{\vee+})^3$.

\newpage
\section{Applications de Gauss et configurations semi-stables}
\label{seGauss}

\subsection{Le bord visuel de $\mathcal I$}
\label{sseBord}

On rappelle que l'immeuble $\mathcal I$ est un espace m\'etrique complet dont on notera la distance $d$. Il est muni de son syst\`eme complet d'appartements. Ainsi tout sous-ensemble convexe isom\'etrique \`a une partie de $\mathbb R^n$ est contenu dans un appartement \cite[11.53]{AB}.

Les r\'esultats suivants r\'esultent essentiellement de ce que $\mathcal I$ est un espace CAT(0) complet. Pour la plupart des d\'emonstrations on se reportera \`a \cite{BH}.

\subsubsection{Rayons et points id\'eaux}
\label{ssseIdeal}

Un rayon (ou une demi-droite) dans $\mathcal I$ est un sous-ensemble $\rho$ isom\'etrique \`a $[0,\infty[$. On confondra dans la suite le rayon et l'isom\'etrie $[0,\infty[\to \mathcal I$. Le point $x=\rho(0)$ est appel\'e l'origine de $\rho$ ou la base. Un rayon est convexe, il est donc contenu dans un appartement $A$ de $\mathcal I$. Et dans $A$, il est de la forme $\{(1-t)x +ty\mid t\geqslant 0\}$ pour $x\ne y$ dans $A$.

On dit que deux rayons $\rho_1$ et $\rho_2$ sont asymptotes (ou parall\`eles) si la fonction (convexe) $t\mapsto d(\rho_1(t),\rho_2(t))$ est born\'ee
. On v\'erifie que ceci d\'efinit une relation d'\'equivalence. Une classe d'\'equivalence de rayons est un point id\'eal de $\mathcal I$ (ou  point \`a l'infini).

Par rapport aux espaces CAT(0) g\'en\'eraux, beaucoup des d\'emonstrations des r\'esultats de cette section sont simplifi\'ees par les faits suivants:

- \'etant donn\'ees deux demi-droites $\rho,\sigma$ il existe des sous-demi-droites $\rho',\sigma'$ contenues dans un m\^eme appartement,

- si de plus $\mathfrak a$ est une alc\^ove ou un germe de quartier et $\rho,\sigma$ sont asymptotes, il existe des demi-droites $\rho_0\subset\rho,\rho_1,...,\rho_n\subset\sigma$ deux \`a deux asymptotes et telles que, $\forall i$, $\mathfrak a$, $\rho_{i}$ et $\rho_{i+1}$ soient contenus dans un m\^eme appartement.

La d\'emonstration du lemme suivant est classique.

\begin{lem}
\label{leRepres}
Soient $x$ un point de $\mathcal I$ et $\xi$ un point \`a l'infini. Alors il existe un unique rayon $\rho$ de base $x$ repr\'esentant $\xi$. On le notera $[x,\xi)$.
\end{lem}

On note $\partial_\infty\mathcal I$ l'ensemble des points \`a l'infini. Soit $F$ une face d'un quartier $\mathfrak Q = x + C$, on note $F_\infty$ l'ensemble des points \`a l'infini $\xi$ tels que $F$ contienne le rayon $[x,\xi)$. Une facette \`a l'infini $\mathfrak f$
est un sous-ensemble de $\partial_\infty\mathcal I$ tel que $\mathfrak f = F_\infty$ pour $F$ une face de quartier.

Les r\'esultats suivant sont classiques
\begin{lem}
\label{leClass}
1) Si $\mathfrak f$ est une facette \`a l'infini et $x$ un point de $\mathcal I$, alors il existe une face de quartier $F$ bas\'ee en $x$ telle que $F_\infty = \mathfrak f$. Ainsi, il y a une bijection entre les facettes \`a l'infini et les faces de quartier bas\'ees en tout point $x$.

2) Deux quartiers de $\mathcal I$ donnent la m\^eme facette \`a l'infini si, et seulement si, ils sont \'equipollents, c'est \`a dire correspondent au m\^eme germe.

3) Les int\'erieurs relatifs des facettes \`a l'infini forment une partition de $\partial_\infty\mathcal I$.
\end{lem}

\subsubsection{Structure d'immeuble}
\label{ssseImmInfini}

On d\'efinit une relation entre les facettes \`a l'infini : $\mathfrak f'$ est une face de $\mathfrak f$ si pour tout point $x$ de $\mathcal I$, la face de quartier $F'$ associ\'ee \`a $\mathfrak f$ est une face de $F$, la face associ\'ee \`a $\mathfrak f$; comme on a consid\'er\'e des facettes ferm\'ees cela \'equivaut \`a $\mathfrak f'\subset \mathfrak f$. Pour un appartement $A$ de $\mathcal I$, on note $A_\infty$ l'ensemble des facettes \`a l'infini donn\'ees par les faces de quartiers de $A$. Cet ensemble est un complexe simplicial et est stable par passage aux faces. En fait, c'est un complexe isomorphe au complexe de Coxeter associ\'e \`a $(W^v,S)$.

Le lemme \ref{leClass} permet de d\'emontrer le th\'eor\`eme suivant.

\begin{thm}
\label{thImmInfini}

Soit $I_\infty$ l'ensemble des facettes \`a l'infini de $\mathcal I$. $I_\infty$ est un immeuble sph\'erique, ses appartements sont en bijection avec ceux de $\mathcal I$. De plus, sa r\'ealisation g\'eom\'etrique est en bijection avec $\partial_\infty\mathcal I$.
\end{thm}

Dans notre cas $\mathcal I$ est l'immeuble de Bruhat-Tits d'un groupe r\'eductif sur un corps non archim\'edien complet, alors $I_\infty$ est l'immeuble de Tits de ce groupe (dont les faces correspondent bijectivement aux sous-groupes paraboliques).

\begin{rem}
On peut \'etendre la topologie de $\mathcal I$ \`a $\overline{\mathcal I}=\mathcal I\amalg \partial_\infty\mathcal I$ en une topologie appel\'ee la topologie conique et alors, si $\mathcal I$ est localement fini,  $\overline{\mathcal I}$ est une compactification de $\mathcal I$. Une base d'ouverts de cette topologie est form\'ee des ouverts de $\mathcal I$ et des ensembles de la forme
$C_x(\xi,\varepsilon) = \{\eta\in\overline{\mathcal I} \mid \eta\ne x,\  \angle_x(\eta,\xi) < \varepsilon \}$
, o\`u $x\in\mathcal I$, $\xi\in\partial_\infty\mathcal I$, $\epsilon >0$ 
 et   $\angle_x(\eta,\xi)$ est l'angle dans un appartement contenant $x$ et le d\'ebut des deux g\'eod\'esiques $[x,\xi)$ et $[x,\eta)$ (si $\eta\in\partial_\infty\mathcal I$ cela co\"{\i}ncide avec la d\'efinition ci-dessous).

Cela dit, la topologie sur $\partial_\infty\mathcal I$ qui va nous int\'eresser est celle de la distance de Tits
\end{rem}

\subsubsection{Angles et distance de Tits}
\label{ssseAngleTits}

Soient $x$ un point de $\mathcal I$, $\xi_1$ et $\xi_2$ deux points \`a l'infini. On pose $\rho_1 = [x,\xi_1)$ et $\rho_2 = [x,\xi_2)$. On consid\`ere le triangle $T(t) = T(x,\rho_1(t),\rho_2(t))$ dans $\mathcal I$ et le triangle $\tilde T(t)=T(\tilde x,\tilde \rho_1(t),\tilde\rho_2(t))$ dans $\mathbb R^2$ "de comparaison", c'est \`a dire dont les c\^ot\'es ont la m\^eme longueur que ceux de $T(t)$ (la condition CAT(0) dit que l'application \'evidente de $\tilde T(t)$ dans $T(t)$ diminue les distances). On note $\tilde\alpha (t)$ l'angle en $\tilde x$ de ce triangle $\tilde T(t)$. Quand $t$ tend vers $0$, $\tilde\alpha(t)$ d\'ecro\^{\i}t contin\^ument et donc la limite existe, on pose
$$
\angle_x(\xi_1,\xi_2) = \lim_{t\to 0}\tilde\alpha(t).
$$ Cette limite est \'egale \`a $\angle_x (\rho_1(t),\rho_2(t))$ d\`es que $t$ est assez petit pour que tous les points de $T(t)$ soient dans un m\^eme appartement. Ainsi, si $T(x,y,z)$ est un triangle g\'eod\'esique dans $\mathcal I$, $\angle_x(y,z)$ est d\'efini de mani\`ere analogue et on a $\angle_x(y,z)\leqslant \angle_{\tilde x}(\tilde y,\tilde z)$, o\`u $T(\tilde x, \tilde y,\tilde z)$ est un triangle de comparaison dans $\mathbb R^2$ (cons\'equence facile de CAT(0)).

Maintenant, on d\'efinit la distance de Tits sur $\partial_\infty\mathcal I$ comme
$$
\angle_{Tits} (\xi_1,\xi_2) = \sup_{x\in\mathcal I} \angle_x (\xi_1,\xi_2).
$$ Par d\'efinition, $\angle_{Tits} (\xi_1,\xi_2)
\geqslant \angle_x (\xi_1,\xi_2)$ pour tous $x\in\mathcal I$, $\xi_1,\xi_2\in \partial_\infty\mathcal I$.

\begin{exem}
Si $\mathcal I$ est un arbre, alors $\angle_x (\xi_1,\xi_2) = 0$ ou $\pi$ et donc $\angle_{Tits} (\xi_1,\xi_2) = \pi$, toujours ! Par contre, en rang sup\'erieur, cette distance peut prendre toutes les valeurs entre $0$ et $\pi$.
\end{exem}

\begin{lem}
\label{leCalcTits}
Soient $\xi,\eta\in\partial_\infty\mathcal I$ et $\rho$ un repr\'esentant de $\xi$. On pose $\varphi(t) = \angle_{\rho(t)}(\xi,\eta)$. Alors, $\lim_{t\to \infty} \varphi(t) = \angle_{Tits}(\xi,\eta)$.
\end{lem}

\begin{thm}
\label{thInfiniComplet}
L'espace m\'etrique $(\partial_\infty\mathcal I,\angle_{Tits})$ est complet.
\end{thm}

\subsection{Des triangles aux configurations semi-stables}\label{sseTrSS}

On cherche \`a r\'ealiser l'\'etape 2 de l'introduction et donc \`a construire une configuration semi-stable associ\'ee \`a un triangle de $\mathcal I$, cf. \cite{KLM2}.

\subsubsection{Fonctions de Busemann}\label{ssseBuse}

On peut plonger $\mathcal I$ dans l'espace $\mathcal F=\mathcal F(\mathcal I,\mathbb R)/\mathbb R$ des fonctions continues sur $\mathcal I$ (avec la topologie de la convergence uniforme sur les born\'es), quotient\'e par les constantes: \`a $x\in\mathcal I$ on associe la fonction $d(x,-)$. On peut montrer \cite[II 8.13]{BH} que $\overline{\mathcal I}$ est hom\'eomorphe \`a l'adh\'erence $\widehat{\mathcal I}$ de $\mathcal I$ dans $\mathcal F$. On ne va utiliser que le plongement de $\partial_\infty\mathcal I$ dans $\mathcal F$:

Soient $x\in\mathcal I$, $\xi\in\partial_\infty\mathcal I$ et $\rho$ un repr\'esentant de $\xi$. L'application $[0,\infty[\to\mathbb R, t\mapsto d(x,\rho(t))-t$ est minor\'ee (par $-d(x,\rho(0))$) et d\'ecroissante; on peut donc noter  $$b_\xi(x)=\lim_{t\to\infty}\;(d(x,\rho(t))-t)$$

Cette fonction $b_\xi$ de $\mathcal I$ dans $\mathbb R$ est la fonction de Busemann associ\'ee \`a $\xi$ (ou plut\^ot \`a $\rho$).

Elle ne d\'epend du choix de $\rho$ qu'\`a une constante pr\`es et est $1-$Lipschitzienne en $x$. On peut montrer que $b_\xi$ est la limite de $\rho(t)$ dans $\mathcal F$. Ainsi la classe de $b_\xi$ dans $\mathcal F$ est en fait dans $\widehat{\mathcal I}$ et ne d\'epend que de $\xi$.
Les lignes de niveau de $b_\xi$ sont les horosph\`eres de centre $\xi$.

\begin{exem}\label{exBuse} Si les points $x$, $y$ et la demi-droite $\rho$ sont dans un m\^eme appartement $A$ et si $\vec\xi={d\over{dt}}(\rho(t))$ est le vecteur directeur unitaire de $\rho$, on a : $b_\xi(x)-b_\xi(y)=\vec{xy}.\vec\xi=-d(x,y)\cos(\vec{yx},\vec\xi)$. Les horosph\`eres (en tout cas leurs intersections avec $A$) sont donc des hyperplans orthogonaux \`a $\rho$.
\end{exem}

\begin{lem}
\label{leDeriv} Soient $\sigma$ une g\'eod\'esique (parcourue \`a vitesse $1$) d'extr\'emit\'e $\eta\in\overline{\mathcal I}$ et $t\in\mathbb R$ tels que $\sigma(t)\not=\eta$, alors la d\'eriv\'ee directionnelle de $b_\xi$ selon $\eta$ en $t$ vaut : ${d\over{dt^+}}(b_\xi\circ\sigma)(t)=-\cos\angle_{\sigma(t)}(\eta,\xi)$.
\end{lem}
\begin{proof}
C'est un exercice facile de g\'eom\'etrie euclidienne dans un appartement contenant le d\'ebut de la g\'eod\'esique de $\sigma(t)$ \`a $\eta$ et la demi-droite $[\sigma(t),\xi[$.
\end{proof}

\begin{lem}
\label{lePente} Soient $\xi,\eta\in\partial_\infty\mathcal I$ et $\sigma$ une demi-droite repr\'esentant $\eta$. La pente asymptotique de $b_\xi$ en $\eta$ est $pente_\xi(\eta)=\lim_{t\to\infty}\;{{b_\xi(\sigma(t))}\over t}$. Elle s'exprime avec la distance de Tits: $pente_\xi(\eta)=-\cos\angle_{Tits}(\xi,\eta)$.
\end{lem}
\begin{proof}
Si $\rho$ est une demi-droite repr\'esentant $\xi$, on v\'erifie que ${{b_\xi(\sigma(t))}\over t}$ ne d\'epend asymptotiquement pas des choix des repr\'esentants $\rho$ et $\sigma$. On peut alors raisonner dans un appartement contenant ces deux g\'eod\'esiques.
\end{proof}

\subsubsection{Configurations pond\'er\'ees et stabilit\'e}\label{ssseConf}

On consid\`ere trois points $\xi_1,\xi_2,\xi_3\in\partial_\infty\mathcal I$ et trois poids $m_1,m_2,m_3\in[0,+\infty[$. On les voit comme une configuration pond\'er\'ee $\psi:(\mathbb Z/3\mathbb Z,\nu)\to\partial_\infty\mathcal I$, o\`u $\nu$ est la mesure sur $\mathbb Z/3\mathbb Z$ de masse $m_i$ en $i$.

La mesure associ\'ee sur $\partial_\infty\mathcal I$ est $\mu=\psi_*\nu=\sum_i\;m_i\delta_{\xi_i}$, de poids total $\vert\mu\vert=m_1+m_2+m_3$. On d\'efinit sa pente: $pente_\mu=-\sum_{i\in\mathbb Z/3\mathbb Z}\;m_i\cos\angle_{Tits}(\xi_i,-)$, c'est une fonction sur $\partial_\infty\mathcal I$.

\begin{dfn}\label{defSS} La mesure $\mu$ et la configuration $\psi$ sont dits semi-stables si la fonction $pente_\mu$ est positive ou nulle sur $\partial_\infty\mathcal I$.
\end{dfn}

Intuitivement cela signifie que les points $\xi_1,\xi_2,\xi_3$ sont \'eloign\'es les uns des autres. Par exemple si $\xi_1=\xi_2=\xi_3$ aucune configuration n'est semi-stable; dans un arbre, si $\xi_1,\xi_2,\xi_3$ sont deux \`a deux diff\'erents (resp. $\xi_1=\xi_2\not=\xi_3$) la configuration est semi-stable si et seulement si $2m_i\leq{}m_1+m_2+m_3$, $\forall i$ (resp. $m_3=m_1+m_2$).

Bien s\^ur une configuration semi-stable le reste si on multiplie tous ses poids par un m\^eme r\'eel positif: les configurations semi-stables forment un c\^one (satur\'e).

La pente peut se r\'einterpr\'eter avec les fonctions de Busemann:
la fonction de Busemann pond\'er\'ee associ\'ee \`a $\mu$ ou $\psi$ est $b_\mu=\sum_i\;m_ib_{\xi_i}$, elle est d\'efinie \`a une constante pr\`es. Si $\eta\in\partial_\infty\mathcal I$ est repr\'esent\'e par $\sigma$, alors le lemme \ref{lePente} dit que : $pente_\mu(\eta)=\lim_{t\to\infty}\;{{b_\mu(\sigma(t))}\over t}$.

\subsubsection{Application de Gauss}\label{ssseGauss}

Soit $T=T(x_1,x_2,x_3)$ un triangle dans $\mathcal I$. On prolonge chaque segment $[x_{i-1},x_i]$ en une demi-droite d'origine $x_{i-1}$, d'extr\'emit\'e not\'ee $\xi_i\in\partial_\infty\mathcal I$ et on consid\`ere les poids $m_i=d(x_{i-1},x_i)$. L'application de Gauss $\psi$ associe donc au triangle $T$ une configuration pond\'er\'ee $\psi_T$ (bien s\^ur il y a plusieurs choix pour $\psi_T$).

\begin{prop}\label{prConfSS}
La configuration $\psi_T$ est semi-stable.
\end{prop}
\begin{proof} Soient $\eta\in\partial_\infty\mathcal I$ et $\gamma_i:[0,m_i]\to[x_{i-1},x_i]$ une param\'etrisation \`a vitesse $1$ de ce segment. D'apr\`es le lemme \ref{leDeriv} et \ref{ssseAngleTits} on a : ${d \over{dt^+}}(b_\eta\circ\gamma_i)(t)=-\cos\angle_{\sigma_i(t)}(\eta,\xi_i)\leq{}-\cos\angle_{Tits}(\eta,\xi_i)$. Ainsi $b_\eta(x_i)-b_\eta(x_{i-1})=\int_0^{m_i}\;{d \over{dt^+}}(b_\eta\circ\gamma_i)(t)dt\leq{}\int_0^{m_i}\;-\cos\angle_{Tits}(\eta,\xi_i)={-m_i}\cos\angle_{Tits}(\eta,\xi_i)$  et donc $0\leq{}\sum_{i\in\mathbb Z/3\mathbb Z}{-m_i}\cos\angle_{Tits}(\eta,\xi_i)$.
\end{proof}

\subsection{Des configurations semi-stables aux triangles}\label{sseSSTr}

On cherche \`a r\'ealiser l'\'etape 3 de l'introduction et donc \`a inverser l'application de Gauss ci-dessus en construisant un triangle dans $\mathcal I$ \`a partir d'une configuration semi-stable $\psi=((\xi_1,m_1),(\xi_2,m_2),(\xi_3,m_3))$, cf. \cite{KLM2}.

\subsubsection{Points fixes}\label{sssePtsFix}

Soient $\xi\in\partial_\infty\mathcal I$ et $t\geq{}0$. On d\'efinit l'application $\phi_{\xi,t}:\mathcal I\to\mathcal I,x\mapsto\rho(t)$ o\`u $\rho$ est le repr\'esentant de $\xi$ issu de $x$. La fonction $t\mapsto d(\phi_{\xi,t}(x),\phi_{\xi,t}(y))$ est convexe, born\'ee et d\'ecroissante.

Si $\psi=((\xi_1,m_1),(\xi_2,m_2),(\xi_3,m_3))$ est une configuration pond\'er\'ee, on lui associe l'application $\phi=\phi_{\xi_3,m_3}\circ\phi_{\xi_2,m_2}\circ\phi_{\xi_1,m_1}$ de $\mathcal I$ dans $\mathcal I$. Par construction $d(x,\phi x)\leq{}\vert\mu\vert=m_1+m_2+m_3$ et $d(\phi(x),\phi(y))\leq{}d(x,y)$.

On cherche un point fixe $x_0$ de $\phi$ qui d\'efinirait un triangle $T=T(x_0,x_1,x_2)$ (avec $x_1=\phi_{\xi_1,m_1}(x_0)$ et $x_2=\phi_{\xi_2,m_2}(x_1)$) v\'erifiant $\psi_T=\psi$.

Pour cela on va utiliser une variante du lemme de point fixe de Bruhat-Tits \cite{KLM2}:

\begin{lem}\label{lePtFix} Soit $\phi:\mathcal I\to\mathcal I$ une application $1-$Lipschitzienne d'un espace CAT(0) complet dans lui m\^eme. S'il existe $x\in\mathcal I$ tel que $\{\phi^n(x)\mid n\geq{}0\}$ est born\'e, alors $\phi$ a un point fixe dans $\mathcal I$.
\end{lem}

\begin{proof} On pose $x_n=\phi^n(x)$ et, pour $y\in\mathcal I$, $r(y)=\limsup_{n\to\infty}\;d(x_n,y)$. Alors $r(\phi y)=\limsup_{n\to\infty}\;d(x_n,\phi y)=\limsup_{n\to\infty}\;d(\phi x_{n-1},\phi y)\leq{}\limsup_{n\to\infty}\;d(x_{n-1},y)=r(y)$; donc $r\circ\phi\leq{}r$. Il suffit donc de montrer que $r$ a un unique minimum dans $\mathcal I$.

On note $\rho=\inf_{\mathcal I}(r)$. Pour $\epsilon>0$ et $y,y'\in\mathcal I$ tels que $r(y),r(y')<\rho+\epsilon$, il existe $n_0$ tel que $d(x_n,y),d(x_n,y')<\rho+\epsilon$, $\forall n\geq{}n_0$. Si $m$ est le milieu du segment $[y,y']$ on a $r(m)\geq{}\rho$, donc $d(x_n,m)>\rho-\epsilon$ pour une infinit\'e de $n\in\mathbb N$. L'in\'egalit\'e (CN) de \cite[3.2.1]{BT} (cons\'equence de la condition CAT(0)) s'\'ecrit alors, pour ces entiers $n$ : $d(y,y')^2\leq{}2(d(x_n,y)^2+d(x_n,y')^2)-4d(x_n,m)^2<16\rho\epsilon$. Ainsi une suite $y_n\in\mathcal I$ telle que $r(y_n)$ tende vers $\rho$ est forc\'ement de Cauchy. Comme $\mathcal I$ est complet, cela montre l'existence et l'unicit\'e d'un $y\in\mathcal I$ tel que $r(y)=\rho$.
\end{proof}

\subsubsection{Le c\^one sur l'immeuble \`a l'infini}\label{ssseCone}

L'immeuble sph\'erique $I_\infty$ de r\'ealisation g\'eom\'etrique sph\'erique $\partial_\infty\mathcal I$ a aussi une r\'ealisation vectorielle (dans le m\^eme sens qu'en \ref{ssseTang}) : c'est le c\^one
$C\partial_\infty\mathcal I$ quotient de $[0,+\infty[\times\partial_\infty\mathcal I$ par la relation qui identifie $\{0\}\times\partial_\infty\mathcal I$ \`a un seul point not\'e $0$. C'est un espace m\'etrique pour la distance $d_C((a,\xi),(b,\eta))^2=a^2+b^2-2ab\cos\angle_{Tits}(\xi,\eta)$.

On montre \cite[p. 60, 61, 188]{BH} que $C\partial_\infty\mathcal I$ est complet, que deux points sont toujours joints par un unique segment g\'eod\'esique et qu'il v\'erifie la condition CAT(0). C'est un immeuble vectoriel.

En fait dans notre cas $C\partial_\infty\mathcal I$ est l'immeuble de Tits d'un groupe r\'eductif, dans sa r\'ealisation vectorielle et on sait qu'il a toutes les propri\'et\'es des immeubles affines plus une, qui caract\'erise les immeubles vectoriels : les appartements sont des espaces vectoriels euclidiens d'origine $0$ (commune \`a tous les appartements).

On peut remplacer, dans ce qui pr\'ec\`ede, $\mathcal I$ par $C\partial_\infty\mathcal I$ et donc consid\'erer $\phi$ sur ce dernier immeuble.

\begin{thm}\label{thFixe} Si la configuration $\psi$ est semi-stable, alors l'application $\phi$ admet un point fixe dans $C\partial_\infty\mathcal I$.
\end{thm}
\begin{proof} \cite[prop. 4.5]{KLM2} Nous sommes dans un immeuble vectoriel que l'on notera $\mathcal I^v$, ses facettes sont des faces de quartier de sommet $0$ et elles correspondent bijectivement (par $F\mapsto F^\infty$) aux facettes de $\partial_\infty\mathcal I=\partial_\infty\mathcal I^v$. On normalise toutes les fonctions de Busemann de fa\c{c}on que $b_\xi(0)=0$. Comme deux facettes sont toujours dans un m\^eme appartement, on peut appliquer le calcul de l'exemple \ref{exBuse} \`a $0$, $x\in\mathcal I^v$ et $\xi\in\partial_\infty\mathcal I$, ainsi $b_\xi(x)=-d(0,x)\cos\angle_0(x,\xi)$. En particulier $\vert b_\xi(x)\vert\leq{}d(0,x)$ et $d(0,x)=\max_{\xi\in\partial_\infty\mathcal I}\;(-b_\xi(x))$.

Si $F$ est une facette de $\mathcal I^v$, on note $F^*$ son \'etoile c'est \`a dire la r\'eunion des facettes (ferm\'ees) contenant $F$. Si $x\in F^*$, $\eta\in F^\infty$ et $\xi\in\partial_\infty\mathcal I$, alors $x$, $\eta$ et $\xi$ sont dans un m\^eme appartement et $\phi_{\xi,t}(x)$ est le translat\'e de $x$ d'une longueur $t$ dans la direction $\xi$.
On a donc $b_\eta(\phi_{\xi,t}(x))-b_\eta(x)=-t\cos\angle_{Tits}(\xi,\eta)$ d'apr\`es \ref{exBuse}.
On d\'efinit l'ensemble $F^{*\circ}$ des $x\in F^*$ tels que la boule $B(x,\vert\mu\vert)$ de $\mathcal I^v$ soit contenue dans $F^*$.
Ainsi, pour $x\in F^{*\circ}$, $\phi_{\xi_1,t}(x)$ (pour $0\leq{}t\leq{}m_1$), $\phi_{\xi_2,t}(\phi_{\xi_1,m_1}(x))$ (pour $0\leq{}t\leq{}m_2$) et $\phi_{\xi_3,t}(\phi_{\xi_2,m_2}(\phi_{\xi_1,m_1}(x)))$ (pour $0\leq{}t\leq{}m_3$) restent dans $F^*$ et on a donc: pour tous $x\in F^{*\circ}$, $\eta\in F^\infty$, $b_\eta(\phi(x))-b_\eta(x)=-\sum\;m_i\cos\angle_{Tits}(\xi_i,\eta)=pente_\mu(\eta)$. Comme $\mu$ est semi-stable, on en d\'eduit que $b_\eta(\phi(x))\geq{}b_\eta(x)$.

Pour appliquer le lemme \ref{lePtFix} on veut montrer que $\phi$ stabilise un born\'e; celui-ci sera l'approximation poly\'edrique d'une boule de centre $0$ que l'on va construire maintenant.

Soit $Q$ un quartier de sommet $0$ dans $\mathcal I^v$, on va construire un sous-ensemble fini $D$ de $Q^\infty$. Pour cela on met d'abord dans $D$ le barycentre du simplexe $Q^\infty$, puis on ajoute successivement des points des faces de $Q^\infty$ (diff\'erentes de $Q^\infty$) ordonn\'ees de fa\c{c}on que la dimension d\'ecroisse: si $F^\infty$ est une telle face, on rajoute au $D_-(F^\infty)$ d\'ej\`a construit un ensemble fini $D(F^\infty)$ de points de l'int\'erieur relatif de $F^\infty$ tel que $F^\infty$ soit recouvert par les boules ouvertes (pour la distance de Tits) de centres ces points de $D(F^\infty)$ et de rayon $(1/3).d(F^\infty,D_-(F^\infty))$. On note $\epsilon(F^\infty)$ la distance de $F^\infty$ au compl\'ementaire de la r\'eunion de ces boules.

Alors le $D$ ainsi construit et $\epsilon=(1/2)\inf\{\epsilon(F^\infty)\mid F^\infty\subset\partial Q^\infty\}$ v\'erifient la condition suivante: si $\eta\in Q^\infty$ et $\zeta\in D$ sont tels que $\angle_{Tits}(\eta,\zeta)\leq{}2\angle_{Tits}(\eta,\zeta')$ pour tous $\zeta'\in D$, alors $\eta$ est \`a distance de Tits $>\epsilon$ de toute face de $Q^\infty$ ne contenant pas $\zeta$.

On consid\`ere l'ensemble $E$ des points de $\partial_\infty\mathcal I$ d'image dans $D$ par la projection de $\partial_\infty\mathcal I$ sur $Q^\infty$ d\'etermin\'ee par les types. Alors la condition de l'alin\'ea pr\'ec\'edent est encore v\'erifi\'ee si l'on change $Q^\infty$ en $\partial_\infty\mathcal I$ et $D$ en $E$.

On consid\`ere la fonction $f=\max_{\zeta\in E}(-b_\zeta)$. On a vu que $\vert f\vert\leq{}d(0,-)$; de plus les boules poly\'edriques $B_f(r)=\{x\in\mathcal I^v\mid f(x)\leq{}r\}$ sont born\'ees: cela se v\'erifie dans le quartier $Q$ et $\{x\in Q\mid -b_\zeta(x)\leq{}r,\;\forall\zeta\in D\}$ est born\'e d'apr\`es l'exemple \ref{exBuse}. Il ne reste donc plus qu'\`a montrer que $\phi$ stabilise $B_f(r)$ pour $r$ assez grand.

Soient $r>0$ et $x\in \mathcal I^v$ tels que $d(0,x)>r$. Montrons que $\forall\zeta\in E$, $-b_\zeta(\phi(x))\leq{}f(x)$.

$-$ Si $\angle_0(x,\zeta)\leq{}2\angle_0(x,\zeta')$, $\forall\zeta'\in E$ et si $F^\infty_\zeta$ est la facette contenant $\zeta$ dans son int\'erieur relatif, alors $x\in F_\zeta^{*\circ}$ pour $r\geq{}\vert\mu\vert/\sin\epsilon$. On a alors  $-b_\zeta(\phi(x))\leq{}-b_\eta(x)\leq{}f(x)$.

$-$ Si $p=\angle_0(x,\zeta)>2\angle_0(x,\zeta')=2q$ pour un $\zeta'\in E$, on peut supposer $\angle_0(x,\zeta')=\min\{\angle_0(x,\zeta'')\mid\zeta''\in E\}$. Alors $\angle_{Tits}(\zeta,\zeta')\leq{}p+q\leq{}3(p-q)$ et, si on note $\theta=\inf\{\angle_{Tits}(\xi,\eta)\mid\xi\not=\eta\in E\}= \inf\{\angle_{Tits}(\xi,\eta)\mid\xi\not=\eta\in D\}>0$ on a $\theta\leq{}p+q\leq{}3(p-q)$. Alors $f(x)=-b_{\zeta'}(x)=-b_{\zeta}(x)+d(0,x)(\cos q-\cos p)=-b_{\zeta}(x)+2d(0,x).\sin{{p+q}\over2}.\sin{{p-q}\over2}\geq{}-b_{\zeta}(x)+\vert\mu\vert$, si $2r\inf\{\sin(\theta/2),\sin(3\pi/4)\}.\sin(\theta/6)\geq{}\vert\mu\vert$. Et alors, comme $b_\zeta$ est $1-$Lipschitzienne et $\phi$ de d\'eplacement au plus $\vert\mu\vert$, on a $-b_\zeta(\phi(x))\leq{}-b_\zeta(x)+\vert\mu\vert\leq{}f(x)$.

On a donc $f(\phi(x))\leq{}f(x)$ si $d(0,x)>r$ avec $r>r_0$ (assez grand). Mais $\phi(B(0,r))\subset B(0,r+\vert\mu\vert)\subset B_f(r+\vert\mu\vert)$. Donc $B_f(r)$ est stable par $\phi$ pour $r>r_0+\vert\mu\vert$.
\end{proof}

\subsubsection{L'argument de transfert}\label{ssseTrans}

Nous venons de trouver un point fixe de $\phi$ dans l'immeuble vectoriel $\mathcal I^v$ correspondant \`a $\partial_\infty\mathcal I$. On a donc dans $\mathcal I^v$ un triangle de longueurs (num\'eriques) de c\^ot\'es $m_1,m_2,m_3$ et de directions de c\^ot\'es $\xi_1,\xi_2,\xi_3\in\partial_\infty\mathcal I$.

Identifions $C^v$ \`a un quartier $Q$ d'origine $0$ de $\mathcal I^v$ et consid\'erons les \'el\'ements $\lambda,\mu,\nu$ de $C^v$ de longueurs num\'eriques  respectives $m_1,m_2,m_3$ et de directions respectives les images de $\xi_1,\xi_2,\xi_3$ dans $Q^\infty$ par la projection de $\partial_\infty\mathcal I$ sur $Q^\infty$. On dira que $\lambda=pr_{C^v}(\xi_1,m_1)$ et, de m\^eme,  $\mu=pr_{C^v}(\xi_2,m_2)$,  $\nu=pr_{C^v}(\xi_3,m_3)$. Ainsi le triangle ci-dessus a pour longueurs de c\^ot\'es $\lambda$, $\mu$ et $\nu$.

On est dans le cadre du corollaire \ref{corSatTrip} et le dernier alin\'ea de celui-ci nous dit qu'il existe dans $\mathcal I$ un triangle de m\^eme longueurs de c\^ot\'es. En effet l'appartement t\'emoin $\mathbb A^v$ de $\mathcal I^v$ s'identifie, avec ses murs et son groupe de Weyl $W^v$, \`a l'appartement t\'emoin $\mathbb A$ de $\mathcal I$, si l'on ne garde dans ce dernier que les murs passant par un sommet sp\'ecial donn\'e.

Cela signifie aussi que $\phi$ a un point fixe dans $\mathcal I$.

\begin{prop}\label{prSat}
Le c\^one $\mathcal T$ du corollaire \ref{corSatTrip} est satur\'e dans $(P^{\vee+})^3$, i.e., s'il existe $(\lambda,\mu,\nu)\in (P^{\vee+})^3$ et $N\in\mathbb N^*$ tels que $(N\lambda,N\mu,N\nu)\in \mathcal T$, alors $(\lambda,\mu,\nu)\in \mathcal T$.
\end{prop}
\begin{proof} S'il existe un triangle de longueurs de c\^ot\'es $N\lambda,N\mu,N\nu$, on lui associe une configuration semi-stable $((\xi_1,m_1),(\xi_2,m_2),(\xi_3,m_3))$ avec  $N\lambda=pr_{C^v}(\xi_1,m_1)$, $N\mu=pr_{C^v}(\xi_2,m_2)$ et  $N\nu=pr_{C^v}(\xi_3,m_3)$, cf. \ref{prConfSS}. la configuration $((\xi_1,m_1/N),(\xi_2,m_2/N),(\xi_3,m_3/N))$ est encore semi-stable (\ref{ssseConf}) et le raisonnement ci-dessus permet de lui associer un triangle dans $\mathcal I$ de longueurs de c\^ot\'es $\lambda$, $\mu$ et $\nu$.
\end{proof}


\bigskip\bigskip\bigskip
\section{Le th\'eor\`eme de saturation}
\label{seSat}

\medskip
On fait le point sur les \'etapes de la d\'emonstration du Th\'eor\`eme \ref{thSat} qui sont d\'ej\`a d\'emontr\'ees. L'\'etape 1) a \'et\'e vue dans la section \ref{seDepli}, voir la remarque  \ref{remHecke}.2.
 Les \'etapes 2) et 3) ont \'et\'e d\'emontr\'ees dans la Section \ref{seGauss}. Il reste \`a montrer les \'etapes 4) et 5).

 On a une configuration semi-stable $\xi=((\xi_1,m_1),(\xi_2,m_2),(\xi_3,m_3))$ avec $\lambda=pr_{C^v}(\xi_1,m_1)$,  $\mu=pr_{C^v}(\xi_2,m_2)$ et $\nu=pr_{C^v}(\xi_3,m_3)$ dans $ P^{\vee+}$; on supposera d\`es le th\'eor\`eme \ref{thFixSom} que $\lambda+\mu+\nu\in Q^\vee$.

\subsection{Facteurs de saturation et action de $P^\vee/Q^\vee$}

On note $\theta$ la plus grande racine et $m_i$ son coefficient sur la racine $\alpha_i$, $\theta = \sum_{i=1}^{l}m_i\alpha_i$. L'alc\^ove fondamentale $\mathfrak a$ est d\'etermin\'ee par les in\'equations :
$$
\alpha_i(v) \geqslant 0,\ \forall i=1,...,l\quad;\quad \theta(v)\leqslant 1.
$$ Si on consid\`ere les poids fondamentaux $\varpi_1,...,\varpi_l$ alors $P^\vee = \oplus\mathbb Z\varpi_i$ et
$$
\mathfrak  a = \{\sum x_i\varpi_i\mid x_i\geqslant 0,\ \sum_ix_i\leqslant 1 \}\ .
$$ Les sommets de $\mathfrak  a$ sont donc $(0,...,0); (1/m_1,0,...,0);\cdots ; (0,...,0,1/m_l)$ dans la base $(\varpi_1,...,\varpi_l)$. De plus, on sait que $P^\vee$ est simplement transitif sur les sommets sp\'eciaux.

\begin{lem}
Le plus petit entier $k\in\mathbb N^*$ tel que, pour tout sommet $s$ de $\mathbb A$, $ks$ est un sommet sp\'ecial est l'entier $k = k_\Phi = ppcm (m_1,...,m_l)$.
\end{lem}

\begin{proof}Si on teste sur les sommets de $\mathfrak  a$, il est clair que $k$ est comme indiqu\'e. Mais $Q^\vee\subset P^\vee$ est simplement transitif sur les alc\^oves donc $Q^\vee\cdot \{$sommets de $\mathfrak  a\} = \{$sommets de $\mathbb A\}$. D'o\`u le r\'esultat.
\end{proof}

\bigskip
On s'int\'eresse maintenant \`a l'action de $P^\vee$ sur $\mathfrak  a$. Soit $\lambda\in P^\vee$, on note $\tau_\lambda$ la translation associ\'ee. On pose $\mathfrak  a' = \tau_\lambda(\mathfrak a)$. Alors il existe un unique $w_\lambda\in W^a$ tel que $w_\lambda \mathfrak  a' = \mathfrak a$. On note $\varphi_\lambda = w_\lambda\circ \tau_\lambda$; si $\lambda\in Q^\vee, \tau_\lambda\in W^a$ et $\varphi_\lambda = Id$. Si $\mu$ est un autre \'el\'ement de $P^\vee$, alors
$$
\varphi_\lambda\varphi_\mu = w_\lambda\tau_\lambda w_\mu\tau_\mu = w_\lambda (\tau_\lambda w_\mu\tau_\lambda^{-1})\tau_\lambda\tau_\mu \in W^a\tau_\lambda\tau_\mu
$$ et $\varphi_\lambda\varphi_\mu (\mathfrak a) = \mathfrak a$. Donc $\varphi_\lambda\varphi_\mu = \varphi_{\lambda+\mu}$. Ce qui donne bien une action de $P^\vee/Q^\vee$ sur l'alc\^ove $\mathfrak  a$.

Comme le groupe est suppos\'e presque simple, cette action se traduit par une permutation des sommets (ou des cloisons) de $\mathfrak  a$, donc par une action sur le diagramme de Dynkin compl\'et\'e. Si on \'etiquette les sommets par les entiers $m_i$ et qu'on met $1$ pour celui qu'on rajoute, on peut remarquer que $P^\vee/Q^\vee$ agit simplement transitivement sur les sommets affect\'es d'un coefficient $1$, c'est-\`a-dire, les sommets sp\'eciaux de l'alc\^ove $\mathfrak  a$.

\bigskip

\noindent{\bf N.B. } $Aut ($Dynkin compl\'et\'e$) = Aut ($Dynkin $) \ltimes (P^\vee/Q^\vee)$.

\bigskip

On a donc des tables faciles \`a \'etablir de $k$ et de $\vert P^\vee/Q^\vee\vert$ (not\'es $k_R$ et $i$ dans \cite[p. 49]{KLM3}).

\subsection{Le point fixe est un sommet}
\label{ssePtFixe}

On sait que l'application 
$\phi = \phi_1\circ\phi_2\circ\phi_3$ associ\'ee \`a la configuration semi-stable $\xi=((\xi_1,m_1),(\xi_2,m_2),(\xi_3,m_3))$ admet un point fixe $x_0$ dans $\mathcal I$. Mais dans un appartement contenant $\xi_1$ et $x$, $\phi_1(x)$ est le translat\'e de $x$ d'une longueur $m_1$ dans la direction $\xi_1$, donc selon un vecteur  $\lambda'\in W^v.\lambda$. Or $\lambda\in P^\vee$, la translation associ\'ee $\tau_{\lambda^^d4}$ envoie donc facettes sur facettes. Ainsi, $\phi_1$ est une application simpliciale. De m\^eme pour $\phi_2$ et $\phi_3$. De plus 
 $\tau_{\lambda^^d4}$ envoie une alc\^ove sur une alc\^ove et permute les types de facettes par l'action de $\overline\lambda\in P^\vee/Q^\vee$; donc, si $\lambda+\mu+\nu\in Q^\vee$, $\phi$ conserve les types. Nous avons  d\'emontr\'e le premier point du th\'eor\`eme suivant, le second est laiss\'e \`a la sagacit\'e du lecteur.

\begin{thm}[Kapovich-Leeb-Millson \cite{KLM3}]
\label{thFixSom}
\noindent
\begin{itemize}
\item Si $\lambda + \mu +\nu \in Q^\vee$ alors $\phi$ fixe un sommet.
\item Si $\phi$ fixe un sommet sp\'ecial alors $\lambda + \mu +\nu \in Q^\vee$.
\end{itemize}
\end{thm}

On a donc un triangle $T= T(x_0,x_1,x_2)$ dans $\mathcal I$ avec $d_{C^v}(x_0,x_1) = \lambda$, $d_{C^v}(x_1,x_2) = \mu$ et $d_{C^v}(x_2,x_0) = \nu$ et $x_0$ un sommet de $\mathcal I$. Comme $\lambda,\mu,\nu\in P^\vee$, $x_1$ et $x_2$ sont aussi des sommets. On peut supposer que $x_0$ et $x_1$ sont dans notre appartement t\'emoin favori $A$. On choisit une alc\^ove $\mathfrak a_0$ de $A$ contenant $x_0$. On r\'etracte le triangle $T$ sur l'appartement $A$ par $\rho_{A,\mathfrak a_0}$. Alors, d'apr\`es les r\'esultats de la section \ref{seDepli}, on obtient un polygone $P(x_0, x_1, y_1,...,y_n, x'_2)$, o\`u $x'_2 =  \rho_{\mathbf a_0,A} (x_2)$ et $P(x_1, y_1,...,y_n, x'_2)$ est un chemin de Hecke de type $\mu$.

Maintenant, on fait dans $A$ une homoth\'etie de centre $0$ et de rapport $k$, alors le polygone $ P(x_0, x_1, y_1,...,y_n, x'_2)$ se transforme en un polygone $P(x'_0, x'_1, y'_1,...,y'_n, x''_2)$, o\`u $x'_0,x'_1,x''_2$ sont des sommets sp\'eciaux de $A$, $d_{C^v}(x'_0,x'_1) = k\lambda$, $d_{C^v}(x_2,x_0) = k\nu$ et $P(x'_1, y'_1,...,y'_n, x''_2)$ est un chemin de Hecke de type $k\mu$. Malheureusement, en g\'en\'eral, ce n'est pas un chemin LS... Mais, on a bien d\'emontr\'e l'\'etape 4).

\subsection{Quand est-ce que Hecke vaut LS ?}

La condition LS est un peu myst\'erieuse par rapport \`a Hecke. On utilise le lemme <<grossier>> suivant.

\begin{lem}
\label{leGrossier}
Soit $\pi:[0,1]\to \mathbb A$ un chemin de Hecke (par rapport \`a $-C^v$) de type $\eta\in C^v\cap P^\vee=P^{\vee+}$. Si les points o\`u $\pi$ est pli\'e sont des sommets sp\'eciaux, alors $\pi$ est LS (et dans ce cas LS $\Leftrightarrow$ Hecke $\Leftrightarrow$ (vrai) billard pli\'e positivement).
\end{lem}

\noindent {\it D\'emonstration. } Pour tout $t$, 
il existe une suite $\xi_0 = \pi'_-(t), \xi_1,...,\xi_m = \pi'_+(t)$ et des racines positives $\beta_1,...,\beta_m$ telles que
\begin{description}
\item[(H1)] $\quad r_{\beta_i}(\xi_{i-1}) = \xi_i $
\item[(H2)] $\quad \beta_i(\xi_{i-1}) <0\ $
\item[(H3)] $r_{\beta_i}\in W^v_{\pi(t)}$, i.e. $\beta_i(\pi(t))\in\mathbb Z$.
\end{description}

On note $w_\pm(t)\in W^v$ l'\'el\'ement de plus petite longueur tel que $\pi'_\pm(t) = w_\pm(t) \eta$. Alors, on a
$$
w_-(t) = w_0> w_1 = r_{\beta_1}w_0 > \cdots w_m = r_{\beta_m}w_{m-1} = w_+(t)\ .
$$ D'apr\`es la proposition \ref{prLS1et2}, il suffit de prouver la condition {\bf (LS3)}. Mais, on sait que pour tout $t$, $w_+(t) \leq w_-(t)$. Quand il y a \'egalit\'e, il n'a rien \`a faire. Quand il n'y a pas \'egalit\'e, comme les points o\`u le chemin se plie sont des sommets sp\'eciaux, on peut \'ecrire
$$
w_+(t) = r_{\beta'_p}\cdots r_{\beta'_1} w_-
$$ avec d\'ecroissance de $1$ des longueurs. Donc $\pi$ est bien LS.
\qed

Malheureusement, il n'est pas clair que, m\^eme apr\`es homoth\'etie, un chemin de Hecke se plie en des sommets sp\'eciaux. C'est la raison des d\'efinitions suivantes.

\subsection{Les chemins LS g\'en\'eralis\'es}

Soit $\eta\in C^v\cap P^\vee$, on choisit une d\'ecomposition $\eta = \eta_1+ \cdots+ \eta_l$ avec 
 $\eta_i\in\mathbb N\varpi_i\subset P^{\vee+}$. Soit $\pi_\eta =\pi_{\eta_1}\star\cdots\star\pi_{\eta_l}$ la concat\'enation des segments $\pi_{\eta_1}=[0,\eta_1]$ et $\pi_{\eta_i}=[\eta_1+...+\eta_{i-1},\eta_1+...+\eta_{i}]$ pour $i\geq2$ (\'evidemment, ce chemin d\'epend de la d\'ecomposition de $\eta$). Soit $\mathfrak a$ une alc\^ove.

\begin{dfn}

Un chemin de Hecke (resp. LS) g\'en\'eralis\'e de type $\eta$ (par rapport \`a $-C^v$ ou \`a $\mathfrak a$) est un chemin $p = p_1\star\cdots\star p_l$ o\`u les $p_i$ sont de Hecke (resp. LS) de type $\eta_i$ (par rapport \`a $-C^v$ ou \`a $\mathfrak a$), $p_i(0)=p_{i-1}(1)$  et pour tout $1\leqslant i\leqslant l$, il existe un vecteur $\xi_i$, une chambre qui contient \`a la fois $\xi_i$ et $(p_i)'_+(0)$ et une $(W^v_{p_i(0)},\vec a_{p_i(0)})-$cha\^{\i}ne de $(p_i)'_-(0)$ \`a $\xi_i$ (o\`u $\vec a_{p_i(0)}$ d\'esigne $-C^v$ ou la chambre d\'efinie dans le th\'eor\`eme \ref{thPliage}).

Dans le cas LS, on suppose de plus que $p_0(0)$ est sp\'ecial.
\end{dfn}

Par exemple, $\pi_\eta$ est un chemin LS g\'en\'eralis\'e de type $\eta$. De plus, on a les propri\'et\'es suivantes (d\'emontr\'ees par Kapovich et Millson et par Littelmann) pour les vrais chemins LS g\'en\'eralis\'es (i.e. par rapport \`a $-C^v$):
\begin{itemize}
\item L'ensemble des chemins LS g\'en\'eralis\'es de type $\eta$ et d'origine $0$ est stable par les op\'erateurs $e_\alpha$ et $f_\alpha$.
\item Le seul chemin LS g\'en\'eralis\'e de type $\eta$ d'origine $0$ et contenu dans $C^v$ est $\pi_\eta$.
\item Tout chemin LS g\'en\'eralis\'e de type $\eta$ est le transform\'e, par des op\'erateurs $f_\alpha$, de $\pi_\eta$.
\item Le th\'eor\`eme de Littelmann sur le produit tensoriel est toujours valable avec des chemins LS g\'en\'eralis\'es de type $\eta$.
\end{itemize}

 On peut faire le lien entre ces chemins  LS g\'en\'eralis\'es et les galeries LS. Il suffit de choisir dans \ref{sseGalChe} la galerie $\gamma_\eta$ contenant le chemin $\pi_\eta$ ci-dessus et non le segment $[0,\eta]$.

\subsection{Conclusion}

On va montrer ici l'\'etape 5. On consid\`ere le polygone $P(x'_0, x'_1, y'_1,...,y'_n, x''_2)$ obtenu \`a la section \ref{ssePtFixe}, o\`u $x'_0,x'_1,x''_2$ sont des sommets sp\'eciaux de $A$, $d_{C^v}(x'_0,x'_1) = k\lambda$, $d_{C^v}(x_2'',x'_0) = k\nu$ et $P(x'_1, y'_1,...,y'_n, x''_2)$ est un chemin de Hecke de type $k\mu$ par rapport \`a une alc\^ove $\mathfrak a_0$ contenant $x_0$. On red\'eplie le chemin dans l'immeuble pour obtenir un triangle sur des sommets sp\'eciaux $T(x'_0, x'_1,z_2)$ de longueurs de c\^ot\'es $(k\lambda, k\mu, k\nu)$. Dans un appartement contenant les sommets $x'_1$ et $z_2$, on remplace le segment $[x'_1,z_2]$ par le chemin $\pi = x'_1 + \pi_{\eta}$, o\`u $\pi_{\eta}$ est le chemin associ\'e ci-dessus \`a une d\'ecomposition de $\eta = k\mu$. Ce chemin n'emprunte que des ar\^etes et donc chaque fois qu'il rencontre un mur, c'est en un sommet.

On r\'etracte sur $A$ par $\rho = \rho_{A,\mathfrak a_0}$. Alors on obtient un polygone $P(x'_0, x'_1, z_1,...,z_m, x''_2)$ tel que 
$\rho\pi = P(x'_1, z_1,...,z_m, x''_2)$ est un chemin qui est pli\'e uniquement en des sommets et de Hecke g\'en\'eralis\'e de type $\eta$ par rapport \`a  $\mathfrak a_0$. C'est clair pour les images des segments de $\pi$. Pour les points anguleux, il faut remarquer que deux ar\^etes d'une m\^eme alc\^ove auront des images dans une m\^eme alc\^ove.

Comme $x'_0$ est sp\'ecial, on peut supposer $x'_0=0$. On note $C^v$ la chambre de Weyl oppos\'ee \`a $\mathfrak a_0$. On replie $A$, en accord\'eon, sur $C^v$ par la projection $pr_{C^v}:A\to C^v$; on obtient ainsi un polygone $P(x'_0=0,x''_1,z'_1,...,z'_m,x'''_2)$ avec $d_{C^v}(x'_0,x''_1)=k\lambda$ (i.e. $x''_1=k\lambda$), $d_{C^v}(x'''_2,x'_0)=k\nu$ (i.e. $x'''_2=k\nu^*$) et $p=P(x''_1,z'_1,...,z'_m,x'''_2)$ est un chemin pli\'e uniquement en des sommets. D'apr\`es le lemme suivant $p$ est de Hecke par rapport \`a $\mathfrak a_0$ et donc de Hecke par rapport \`a $-C^v$ (remarque \ref{remHecke}.2).

Enfin, on applique l'homoth\'etie de centre $x'_0$ et de rapport $k$, on obtient un chemin de Hecke g\'en\'eralis\'e dont les coudes sont des sommets sp\'eciaux. Par le lemme \ref{leGrossier}, ce chemin est LS g\'en\'eralis\'e; le th\'eor\`eme de d\'ecomposition \ref{thDecomp} s'applique et prouve que  $\big (V(k^2\lambda)\otimes V(k^2\mu)\otimes V(k^2\nu)\big )^{G^\vee} \ne \{0\}$. On a bien achev\'e  l'\'etape 5) et donc la d\'emonstration du th\'eor\`eme de saturation \ref{thSat}

\begin{lem} Soit $\mathfrak a$ l'alc\^ove de $A$ de sommet $0$ et oppos\'ee \`a la chambre de Weyl $C^v$. Soit $p$ un chemin de Hecke g\'en\'eralis\'e (ou non g\'en\'eralis\'e) de type $\eta$ par rapport \`a $\mathfrak a$ dans $A$. Alors le chemin repli\'e $pr_{C^v}\circ p$ est de Hecke g\'en\'eralis\'e (ou non g\'en\'eralis\'e) de type $\eta$ par rapport \`a $\mathfrak a$ dans $C^v$.
\end{lem}

\begin{proof} Le chemin $pr_{C^v}\circ p$ s'obtient \`a partir de $p$ par une suite de pliages r\'etractant $A$ sur un demi-appartement contenant $C^v$ et de mur contenant $0$. Soient donc $M$ un mur contenant $0$, $D$ le demi-appartement limit\'e par $M$ contenant $C^v$ et $\pi_D$ le pliage de $A$ sur $D$. On va montrer que $\pi_D\circ p$ est de Hecke (g\'en\'eralis\'e) par rapport \`a $\mathfrak a$. Cela se v\'erifie en chaque point $p(t)$ de $p$. Si $p(t)\notin M$, alors $\vec a_{p(t)}$ contient $\mathfrak a$ et sa sym\'etrique $s_M(\mathfrak a)$; de plus, au voisinage de $t$, $\pi_D\circ p$  est \'egal \`a $p$ ou \`a $s_M\circ p$. Donc $\pi_D\circ p$ v\'erifie encore la condition locale impos\'ee.

Supposons $p(t)\in M$. Par hypoth\`ese il existe une $(W^v_{p(t)},\vec a_{p(t)})-$cha\^{\i}ne de $p'_-(t)$ \`a $\xi$ o\`u $\xi$ est dans une m\^eme chambre que $p'_+(t)$. Donc $\pi_D(\xi)$ et $\pi_D(p'_+(t))=(\pi_D\circ p)'_+(t)$ sont dans une m\^eme chambre et du m\^eme c\^ot\'e de $M$. Ainsi $\pi_D(\xi)$ est \'egal \`a $\xi$ ou $s_M(\xi)$ avec comme positions:\quad$\xi,\vec a_{p(t)}\;\mid_M\;s_M(\xi)$. De m\^eme $(\pi_D\circ p)'_-(t)$ est \'egal \`a $p'_-(t)$ ou \`a $s_M(p'_-(t))$ avec cette fois :
\quad $s_M(p'_-(t)),\vec a_{p(t)}\;\mid_M\;p'_-(t)$. Ainsi en compl\'etant \'eventuellement, par le d\'ebut et/ou la fin, la $(W^v_{p(t)},\vec a_{p(t)})-$cha\^{\i}ne de $p'_-(t)$ \`a $\xi$, on obtient une $(W^v_{p(t)},\vec a_{p(t)})-$cha\^{\i}ne de $(\pi_D\circ p)'_-(t)$ \`a $\pi_D(\xi)$.
\end{proof}


\medskip

Institut \'Elie Cartan, Unit\'e Mixte de Recherche 7502, Nancy-Universit\'e, CNRS

Boulevard des aiguillettes, BP 70239, 54506 Vand\oe uvre l\`es Nancy Cedex (France)

Stephane.Gaussent@iecn.u-nancy.fr ; Nicole.Bardy@iecn.u-nancy.fr ;

Cyril.Charignon@iecn.u-nancy.fr ; Guy.Rousseau@iecn.u-nancy.fr


\bigskip\bigskip\bigskip\bigskip
\begin{thebibliography}{BGG2}

\bibitem{AB}
P. Abramenko et K. S. ~Brown, {\em Buildings: Theory and applications},  Grad. Texts in Math. {\bf 248}, Springer-Verlag, New-York (2008).

\bibitem{BaGa}
P. Baumann et S. Gaussent, {\em On Mirkovi\'c-Vilonen cycles and crystal combinatorics}, Represent. Theory {\bf 12} (2008), pp 83-130.

\bibitem{BH}
M. Bridson et A. Haefliger, {\em Metric spaces of non positive curvature}, Grundlehren der math. Wiss., {\bf 319}, Springer (1999).

\bibitem{B}
K. S. ~Brown, {\em Buildings},  Springer-Verlag, New-York (1989).

\bibitem{BT}
F. ~Bruhat et J. ~Tits, {\em Groupes r\'eductifs sur un corps local, I}, Publ. Math. I.H.E.S., {\bf 41}, (1972), pp. 5-252.

\bibitem{BT2}
F. ~Bruhat et J. ~Tits, {\em Groupes r\'eductifs sur un corps local, II}, Publ. Math. I.H.E.S., {\bf 60}, (1984), pp. 5-184.

\bibitem{GL}
S. Gaussent et P. Littelmann, {\em LS-galleries, the path model, and MV-cycles}, Duke Math. J. {\bf 127}, (2005), pp. 35--88.

\bibitem{GR}
S. Gaussent et G. Rousseau, {\em Kac-Moody groups, hovels and Littelmann paths}, Annales Inst. Fourier {\bf 58}, (2008), pp. 2605-2657.

\bibitem{KLM1}
M. Kapovich, B. Leeb et J.J. Millson, {\em Convex functions on symmetric spaces, side lengths of polygons and stability inequalities for weighted configurations at infinity}, J. Diff. Geometry. {\bf 81}, (2009), pp. 297-354.

\bibitem{KLM2}
M. Kapovich, B. Leeb et J.J. Millson, {\em Polygons in buildings and their refined side lengths}, Geometric And Functional Analysis, \`a para\^{\i}tre.

\bibitem{KLM3}
M. Kapovich, B. Leeb et J.J. Millson, {\em Polygons in symmetric spaces and buildings with applications to algebra}, Memoir Amer. Math. Soc. {\bf 896} (2008).

\bibitem{KM}
M. Kapovich et J.J. Millson, {\em A path model for geodesics in euclidean buildings and its applications to representation theory}, Geometry, Groups and Dynamics {\bf 2}, (2008), pp. 405-480.


\bibitem{KT}
A. Knutson et T. Tao, {\em The honeycomb model of $GL_n(\mathbb C)$ tensor products. I Proof of the saturation conjecture}, J. Amer. Math. Soc. {\bf 12}, (1999), pp. 1055-1090.

\bibitem{L1}
P. Littelmann, {\em A Littllewood-Richardson rule for symmetrizable Kac-Moody algebras}, Inventiones Math. {\bf 116}, (1994), pp. 329-346.

\bibitem{L2}
P. Littelmann, {\em Paths and root operators in representation theory}, Annals of Math. {\bf 142}, (1995), pp. 499-525.

\bibitem{MV}
I.~Mirkovi\'c and K.~Vilonen, \textit{Perverse sheaves on affine
Grassmannians and Langlands duality,} Math.\ Res.\ Lett.\ \textbf7
(2000), 13--24.

\bibitem{R}
M.~Ronan, {\em Lectures on Buildings},  Academic Press, (1989).

\bibitem{T74}
J.~Tits, {\em Buildings of spherical
type and finite $BN-$pairs, Lecture notes in Math. 386},  Springer-Verlag, Heidelberg (1974).




\end{thebibliography}
\end{document}